\documentclass{amsart}
\usepackage{amsmath,amssymb}
\usepackage{pstricks, pst-plot, pst-node, pst-grad, pst-math}
\usepackage{graphicx,xcolor}
\usepackage{arydshln}

 \newtheorem{thm}{Theorem}[section]
 \newtheorem{cor}[thm]{Corollary}
 \newtheorem{lem}[thm]{Lemma}
 \newtheorem{prop}[thm]{Proposition}
 \theoremstyle{definition}
 \newtheorem{defn}[thm]{Definition}
 \theoremstyle{remark}
 \newtheorem{rem}[thm]{Remark}
 
 \numberwithin{equation}{section}
 \usepackage{xcolor}
 \usepackage{mathabx}

\numberwithin{equation}{section}

\newcommand\CL{{\mathcal L}}

\newcommand\CK{{\mathcal K}}

\newcommand{\dom}{\mathop{\rm dom}}
\newcommand{\Ran}{\mathop{\rm ran}}

\newcommand{\N}{{\mathbb N}}

\newcommand{\R}{{\mathbb R}}
\newcommand{\C}{{\mathbb C}}
\newcommand{\Z}{{\mathbb Z}}
\newcommand{\T}{{\mathbb{T}}}

\newcommand\wh{\widehat}

\newcommand\Bs{\mathcal{B}}
\newcommand\Ls{\mathcal{L}}
\newcommand\Ks{\mathcal{K}}

\newcommand\G{\Gamma}
\newcommand\Gh{\widehat{\Gamma}}

\newcommand\Kk{\widetilde{K}}

\newcommand{\bet}{\widehat{\beta}}
\newcommand{\alp}{\widehat{\alpha}}
\newcommand{\Hs}{\mathcal{H}}

\newcounter{marke}
\newcommand{\bl}{\begin{list}{\roman{marke})}{\usecounter{marke}
\topsep 0 cm \itemsep 0cm}}
\newcommand{\el}{\end{list}}

\newcommand\be{\begin{equation}}
\newcommand\ee{\end{equation}}

\newcommand{\Ltwo}{L^{2}(\mathbb{T}^{n})}

\newcommand{\eC}{\overline{\C}}

\newcommand{\domain}{\mathcal{C}}

\newcommand{\bb}{b}
\newcommand{\cc}{c}

\newcommand{\ac}{\alpha_0}
\newcommand{\oc}{\omega_0}
\newcommand{\alc}{\alpha'}
\newcommand{\bec}{\beta'}
\newcommand\ip[2]{\langle #1, #2 \rangle}

\begin{document}

\title[Accumulation of Complex Eigenvalues\ ]{Accumulation of Complex Eigenvalues of a Class of Analytic Operator Functions }

\author{Christian Engstr\"om}  \address{Department of Mathematics and Mathematical Statistics, Ume\aa \ University, SE-901 87 Ume\aa, Sweden}
\email{christian.engstrom@math.umu.se}

\author{Axel Torshage}  
\address{Department of Mathematics and Mathematical Statistics, Ume\aa \ University, SE-901 87 Ume\aa, Sweden}
\email{axel.torshage@math.umu.se}

\subjclass{47A56, 47J10, 47A10, 47A12}
\keywords{Operator pencil, spectral divisor, numerical range, non-linear spectral problem}

\begin{abstract}
For analytic operator functions, we prove accumulation of branches of complex eigenvalues to the essential spectrum. Moreover, we show minimality and completeness of the corresponding system of eigenvectors and associated vectors. These results are used to prove sufficient conditions for eigenvalue accumulation to the poles and to infinity of rational operator functions. Finally, an application of electromagnetic field theory is given.
\end{abstract}

\maketitle
\section{Introduction} 
 In recent years quantitative information on the discrete spectrum of non-self\- adjoint operators has gained considerable interest \cite{MR1819914,MR2540070,MR2820160,MR3054310,MR3556444}. In particularly, Pavlov's influential papers on accumulation of complex eigenvalues to the essential spectrum of Schr\"odinger operators with non-selfadjoint Robin boundary conditions \cite{MR0203530,MR0234319,MR0348554} have been extended to magnetic Schr\"odinger operators \cite{MR3682742} and to Schr\"odinger operators with a complex potential \cite{MR3627408}. Importantly, the potential will in some cases also depend on time \cite{MR2038194} and differential operators with time dependent coefficients are common in e.g. electromagnetics \cite{MR1409140} and viscoelasticity \cite{Sanchez-Palencia}. In these cases, theory for operator functions with a non-linear dependence of the spectral parameter is used to determine the spectral properties. Krein \& Langer \cite{KL78} proved for selfadjoint quadratic operator polynomials $\lambda-A-\lambda^2 B$ with real numerical range, sufficient conditions for eigenvalue accumulation in terms of the numerical range of $A$ and the numerical range of $B$. More recently, selfadjoint operator functions with real eigenvalues have been studied extensively \cite{MR2216445,MR2397852,MR2931941,MR3460233,MR3543766}.  Still, there have been very few results on accumulation of complex eigenvalues of operator functions with the notable exception \cite{MR1911850} that proved accumulation of eigenvalues for a quadratic operator polynomial from elasticity theory.

In this paper, we study polynomial and rational operator functions. These types of functions share spectral properties with a linear non-selfadjoint operator called the linearized operator. However, since in non-trivial cases the linearized operator is not a relatively compact perturbation of a selfadjoint operator, the known results for non-selfadjoint operators can not be used to prove accumulation of eigenvalues. Therefore, we will explore factorization results for holomorphic operator functions \cite[Chapter III]{ASM}. We  extend theory based on those factorization results with the aim to provide sufficient conditions for eigenvalue accumulation of a class of unbounded rational operator functions. A major difficulty is that in theory based on the factorization of operator polynomials one must know specific properties of the numerical range to prove accumulation of eigenvalues. The main contribution of \cite{MR1911850} is that they for the particular quadratic operator polynomial prove the existence of a bounded part of the numerical range that is separated from rest of the numerical range.
In  \cite{ATEN}, we proposed a new type of enclosure of the numerical range that can be used to determine the number of components of the numerical range. Importantly, it is much easier to determine the number of components of this enclosure than to directly determine the number of components of the numerical range. Therefore, the new enclosure of the numerical range is used to prove accumulation of eigenvalues, and to prove completeness of the corresponding system of eigenvectors and associated vectors.

The paper is organized as follows. In Section $2$, we present the basic notation and definitions used in the paper. 

In Section $3$, we consider for a class of bounded analytic operator functions minimality and completeness of the set of eigenvectors and associated vectors corresponding to a branch of eigenvalues.  In particular, we generalize \cite[Theorem 22.13]{ASM} to cases with a spectral divisor of order larger than one. Our main results are Theorem \ref{main} and Theorem \ref{main2}.

In Section $4$, we study rational operator functions whose values are Fredholm operators. The main results are Theorem \ref{prop:mult} and Theorem \ref{prop:inf}, which utilize the results presented in Section $3$ to show accumulation of complex eigenvalues to the poles and to complex infinity.

In Section $5$, we take advantage of the enclosure of the numerical range introduced in \cite{ATEN} and study a particular class of unbounded rational operator functions in detail.  The main results are Theorem \ref{3lemallB}, Theorem \ref{3lemallBi}, and Theorem \ref{3lemallBt} that state explicit sufficient conditions for accumulation of complex eigenvalues to the poles, and completeness of the corresponding set of eigenvectors and associated vectors. 

In Section $6$, we apply our abstract results to a rational operator function with applications in electromagnetic field theory.

\section{Preliminaries}
In this section we introduce the notation and the operator theoretic framework used in the rest of the paper. 

Throughout this paper is $\Hs$ a separable Hilbert space with inner product $\ip{\cdot}{\cdot}$ and norm $\|\cdot\|$. Let $\CL (\Hs)$ denote the collection of closed linear operators on $\Hs$ and denote by $\sigma(A)$ the spectrum of $A\in\CL (\Hs)$. The essential spectrum $\sigma_{ess}(A)$ is defined as the subset of $\sigma(A)$ where $A-\lambda$ is not a Fredholm operator and the discrete spectrum $\sigma_{disc}(A)$ is the set of all isolated eigenvalues of finite multiplicity. Moreover, the sets $\sigma_p(A)$, $\sigma_c(A)$, and $\sigma_r(A)$ are the point-, continuous-, and residual spectrum, respectively. With $\Ran A$ we denote the range of $A$ and $\ker A$ denotes the kernel of $A$. Further if $A$ is unbounded we let $\dom A$ denote the domain of $A$.

Let $\Bs(\Hs)$ and $\Ks (\Hs)$ denote the spaces of bounded and compact operators, respectively. The Schatten-von Neumann class is defined as
\[
	S_p(\Hs):=\left\{A\in \Ks(\Hs):\sum_{i=1}^{\infty}s_i(A)^p<\infty\right\},
\]
where $s_i(A)$ are the singular values ($s$-numbers) of $A$ and $1\leq p <\infty$.  

Assume that $\Hs= \Hs_1\oplus\Hs_0$ is the orthogonal sum of two Hilbert spaces and let $\Pi_i:\Hs\rightarrow \Hs$, $i=0,1$ be the orthogonal projection on $\Hs_i$. Then, the inner product on $\Hs_i$ is  $\ip{\cdot}{\cdot}_{\Hs_i}=\ip{V_i\cdot}{V_i\cdot}$, where  $V_i^*:\Hs\rightarrow\Hs_i$ is defined by the partial isometry
\begin{equation}\label{1partis}
V_i^*u=\bigg\{\begin{array}{l l}
u,&u\in\Hs_i\\
0,&u\in\Hs_i^\bot
\end{array},
\end{equation}
which implies $V_i^*V_i=I_{\Hs_i}$ and $V_iV_i^*=\Pi_i$. 
Let $A\in\Bs(\Hs)$, then $A\in\Bs(\Hs_1\oplus\Hs_0)$ has the block operator matrix representation
\begin{equation}\label{eq:BlockRep}
	A=\begin{bmatrix}
	V_1^*AV_1&V_1^*AV_0\\
	V_0^*AV_1&V_0^*AV_0\\
	\end{bmatrix}.
\end{equation}

An operator function $T:\domain\rightarrow\CL (\Hs)$ with domain $\dom T$ is defined on a set $\domain\subset\C$ and take values in $\CL (\Hs)$. The spectrum of $T$ is defined as
\begin{equation}
	\sigma (T):=\{\lambda\in\domain\,:\, 0\in\sigma(T(\lambda))\}
\end{equation}
and $\lambda\in\domain$ is called an eigenvalue of $T$ if there exists an $u\in\dom T\setminus\{0\}$ such that $T(\lambda)u=0$. Let $P:\C\rightarrow\Bs(\Hs)$ denote an analytic operator function and assume that for some $\lambda\in\C$ there are $\{u_i\}_{i=0}^{m-1}\in \Hs$ such that 
\[
	\sum_{i=0}^j\frac{1}{i!}P^{(i)}(\lambda)u_{j-i}=0,\quad j=0,\hdots,m-1.
\] 
Then $\{u_i\}_{i=0}^{m-1}$ is said to be a Jordan chain of length $m$ at $\lambda$, \cite[\textsection 11]{ASM}. 

In the paper we will use the concepts of equivalence for bounded operator functions. The bounded operator functions $P:\domain\rightarrow\Bs(\Hs)$ and $R:\domain\rightarrow\Bs(\Hs)$ are called equivalent on $\domain\subset\C$ if there exist bounded operators $E(\lambda)$ and $F(\lambda)$ invertible for all $\lambda\in \domain$ such that $P(\lambda)=E(\lambda)R(\lambda)F(\lambda)$, $\lambda\in\domain$. Let $I_{\Hs_0}$ denote the identity operator on $\Hs_0\subset\Hs$. If $R$ and $C\oplus I_{\Hs_0}$ are equivalent on $\domain$, then $R$ is said to be equivalent to $C$ on $\domain$ after extension, \cite{MR0482317,ATEQ}.

Let $P$ be a bounded operator polynomial, then the bounded operator polynomial $R$ is called a spectral divisor of order $k$ if there exists a bounded operator polynomial $Q:\C\rightarrow \Bs(\Hs)$ such that
\begin{equation}\label{spdiv}
	P(\lambda)=Q(\lambda)R(\lambda)=\left(\sum_{i=0}^{M-k}\lambda^iQ_i\right)\left(\lambda^k+\sum_{i=0}^{k-1}\lambda^iR_i\right),
\end{equation}
where, $\sigma(R)= \G\cap \sigma(P)$ and $\sigma(Q)\subset \C\setminus \overline{\G}$ for some open $\G\subset \C$. Note that if $R$ is a spectral divisor of $P$, then $R$ is equivalent to $P$ on $\G$. If $k=1$, the operator $-R_0$ is called a spectral root of $P$ on $\G$.

The numerical range of an operator function $T:\domain\rightarrow\CL (\Hs)$ with domain $\dom T$ is the set
\begin{equation}
	W (T):=\{\lambda\in\domain\,:\, \exists u\in\dom T(\lambda)\setminus\{0\}, \ip{T(\lambda)u}{u}=0\},
\end{equation}
which is disconnected in general.

\section{Bounded analytic operator functions}
Consider the bounded operator function $P:\C\rightarrow\Bs(\Hs)$ defined as
\begin{equation}\label{Adef2}
P(\lambda):=\sum_{i=0}^M \lambda^i P_i,\quad P_0=(I_{\Hs}+K)H, \quad P_k=I_{\Hs}+\Kk ,
\end{equation}
where $M\in\N\cup\{\infty\}$,  $k\in \N$, and $k\leq M$. Assume that $K$, $\Kk$, and $P_i$ for $i=1,\hdots k-1$ are compact and $P$ has a spectral divisor $R$ of order $k$  on some open set $\G\subset \C$ with $0\in\G$. Further $H\in S_p(\Hs)$ is normal and has its spectrum located on a finite number of rays from the origin. 
\begin{lem}\label{1lemsd}
Let $R$ denote the spectral divisor of order $k$ of $P$.
 Then $R_i$ is compact for $i=0,\hdots,k-1$ and $R_0$ can be written in the form $R_0=(I_{\Hs}+\widehat{K})H$ for some compact $\widehat{K}$. Furthermore, $I_{\Hs}+\widehat{K}$ is invertible if and only if $I_{\Hs}+K$ is invertible and
\begin{equation}\label{1samker}
	 \bigcap_{j=0}^{i}\ker P_j\subset \ker R_i.
\end{equation}
\end{lem}
\begin{proof}
In the proof, we consider the equality $P(\lambda)=Q(\lambda)R(\lambda)$, $\lambda\in\C$ term-wise. From the assumptions follow $\sigma(Q)\cap \G=\emptyset$ and thus $0\notin\sigma(Q)$, which implies that
\begin{equation}\label{Cprops}
	\begin{array}{l}
	R_0=Q_0^{-1}P_0,\\
	R_1=Q_0^{-1}\left(P_1-Q_1R_0\right),\\
	R_2=Q_0^{-1}\left(P_2-Q_1R_1-Q_2R_0\right),\\
	\qquad \vdots\\
	R_{k-1}=Q_0^{-1}\left(P_{k-1}-\sum_{i=1}^{k-1}Q_iR_{k-1-i}\right),\\
	Q_0=I_{\Hs}+\Kk -\sum_{i=1}^{k}Q_iR_{k-i}.\\
	\end{array}
\end{equation}
Since $P_i$ is compact for $i=0,\hdots,k-1$, it follows from \eqref{Cprops} that $R_i$, $i=0,\hdots,k-1$ are compact and that the inclusion \eqref{1samker} holds. This yields that $Q_0$ is the sum of the identity operator and a compact operator, which implies $Q_0^{-1}=I_{\Hs}+K'$ where $K'$ is compact, and thus
\[
	R_0=(I_{\Hs}+K')(I_{\Hs}+K)H=(I_{\Hs}+\widehat{K})H.
\]
Finally the invertibility of $I_{\Hs}+K'$ implies that $I_{\Hs}+\widehat{K}$ is invertible if and only if $I_{\Hs}+K$ is invertible.
\end{proof}

Set $\Hs_0:=\ker H$, $\Hs_1:=\Hs\bot\Hs_0$, and let $V_i$ for $i=0,1$ denote the partial isometry \eqref{1partis}. Since $H$ is a normal operator that vanish on $\Hs_0$, the operator 
$H\in \Bs(\Hs_1\oplus\Hs_0)$ has the representation 
\begin{equation}\label{H1norm}
	H=\begin{bmatrix}
	H_1& 0\\
	0& 0
	\end{bmatrix}=\begin{bmatrix}
	V_1^*HV_1& 0\\
	0& 0
	\end{bmatrix},
\end{equation}
where $H_1\in S_p(\Hs_1)$ is normal  and $\ker H_1=\{0\}$. 
 In the following we use the notation 
\begin{equation}\label{eq:Lem32}
	I_{\Hs_1}+K_1=V_1^*Q_0^{-1}(I_\Hs+K)V_1,
\end{equation}
where  $K_1$ is compact.
Moreover, the block operator matrix representation of the spectral divisor $R:\C\rightarrow\Bs(\Hs_1\oplus\Hs_0)$. 

\begin{lem}\label{1lemext}
Assume that the operators in \eqref{Adef2} satisfy $\ker H\subset \ker P_i$ for all $i=1,\hdots, k-1$. Then the spectral divisor $R$ is equivalent to $C:=V_1^*R V_1$ on $\C\setminus\{0\}$ after extension and $C$ can be written in the form
\begin{equation}\label{Gdef2}
C(\lambda):=\lambda^k+\sum_{i=0}^{k-1} \lambda^iC_i,\quad C_0=(I_{\Hs_1}+K_1)H_1,
\end{equation}
where $C_i$ is compact for $i=0,\hdots,k-1$ and $I_{\Hs_1}+K_1$ is given by \eqref{eq:Lem32}. Moreover, $H_1\in S_p(\Hs_1)$ is a normal  operator with $\ker{H_1}=\{0\}$ and $\sigma (H_1)$ is located on a finite number of rays from the origin.
\end{lem}
\begin{proof}
By the assumption  that $\ker H\subset \ker P_i$ for all $i=1,\hdots, k-1$, it follows from \eqref{1samker} that $\ker H\subset \ker R_i$ for all $i=1,\hdots, k-1$. Hence, the operator function $R$ on $\Hs_1\oplus\Hs_0$ has the representation
\begin{equation}\label{1Rblock}
	R(\lambda):=\begin{bmatrix}
	\lambda^{k}+ \sum_{i=0}^{k-1} \lambda^i V_1^* R_i V_1&\\
	\sum_{i=0}^{k-1} \lambda^i V_0^*R_i V_1&\lambda^k
	\end{bmatrix}=\begin{bmatrix}
	C(\lambda)&\\
	\sum_{i=0}^{k-1} \lambda^i D_i&\lambda^k
	\end{bmatrix}.
\end{equation}
Then the identity
\begin{equation}\label{1projeq}
	\begin{bmatrix}
	C(\lambda)&0\\
	\sum_{i=0}^{k-1} \lambda^i D_i&\lambda^k
	\end{bmatrix}=\begin{bmatrix}
	C(\lambda)&0\\
	0&I_{\Hs_0}
	\end{bmatrix}\begin{bmatrix}
	I_{\Hs_1}&0\\
	\sum_{i=0}^{k-1} \lambda^i D_i&\lambda^k
	\end{bmatrix},
\end{equation}
implies that $C$ is equivalent to $R$ after extension. Since $V_1^* V_1$ is the identity on $\Hs_1$, the representation \eqref{Gdef2} follows from Lemma \ref{1lemsd}.
\end{proof}

\begin{cor}\label{eigcor}
Assume that the operators in \eqref{Adef2} satisfy $\ker H\subset\ker P_i$ for all $i=1,\hdots, k-1$. Then
$\lambda \in \sigma(P)\cap \G\setminus\{0\}$ if and only if $\lambda\in\sigma(C)\setminus \{0\}$. 
Further $\{u_i\}_{i=0}^{m-1}$ is a Jordan chain of $C(\lambda)$ of order $m$ at $\lambda\in \G\setminus\{0\}$ if and only if $\{v_i\}_{i=0}^{m-1}$ is a Jordan chain of $P(\lambda)$ of order $m$, where $D_i$ is given by \eqref{1Rblock} and
\begin{equation}\label{eigcorv}
	v_i=	u_i	-\sum_{j=0}^{i}\frac{(-1)^{j}}{j!} \left(\sum_{l=0}^{k-1}\frac{(k+j-1-l)!}{(k-1-l)!} \lambda^{l-k-j} D_l\right)u_{i-j}.
\end{equation}
\end{cor}
\begin{proof}
Since $R$ is a spectral divisor of $P$ on $\G$, it follows that $\lambda \in \sigma(P)\cap \G\setminus\{0\}$ if and only if $\lambda\in\sigma(R)\setminus\{0\}$. The correspondence of the Jordan chains follows from \cite[Proposition 1.2]{MR1155350}. Using the equivalence \eqref{1projeq} it can then be shown that $\{u_i\}_{i=0}^{m-1}$ is a Jordan chain of $C(\lambda)$ if and only if $\{v_i\}_{i=0}^{m-1}$ is a Jordan chain of $R(\lambda)$ where
\[
	v_i=\begin{bmatrix}
	u_i\\
	-\sum_{j=0}^{i}\dfrac{(-1)^{j}}{j!} \left(\sum_{l=0}^{k-1}\dfrac{(k+j-1-l)!}{(k-1-l)!} \lambda^{l-k-j} D_l\right)u_{i-j}
	\end{bmatrix},
\]
which is \eqref{eigcorv} represented in $\Hs_1\oplus\Hs_0$.
\end{proof}

Lemma \ref{1lemext} and Corollary \ref{eigcor} imply that if the polynomial \eqref{Adef2} has a spectral divisor, then we can obtain its spectral properties from $C$, which has suitable properties. 

\begin{rem}
In the important special case  $k=1$, the operator $Z:=-R_0$ is a spectral root of $P$. Then the operator function $P$ shares spectrum with the operator $Z$ on $\G\setminus\{0\}$, which  simplifies the expression of the Jordan chains \eqref{eigcorv} to
\begin{equation*}\label{eigcorvsimp}
	v_i=	u_i	-\sum_{j=0}^{i}\frac{(-1)^{j}}{\lambda^{j+1}}  D_0u_{i-j}.
\end{equation*}
\end{rem}

\begin{lem}\label{redker}
Assume that the operators in \eqref{Adef2} satisfy $\ker H\subset \ker P_i\cap \ker \Kk$ for all $i=1,\hdots, k-1$. Let $C$ denote the polynomial \eqref{Gdef2}. Then, $I_{\Hs_1}+V_1^*KV_1$ is invertible if and only if $I_{\Hs_1}+K_1$ is invertible.
\end{lem} 
\begin{proof}
From the representation of $Q_0$ in \eqref{Cprops} and \eqref{1Rblock}  it follows that $Q_0\in \Bs(\Hs_1\oplus\Hs_0)$ can be written in the form
\[
	Q_0=\begin{bmatrix}
	I_{\Hs_1}-K^{(1)}&\\
	-K^{(2)}&I_{\Hs_0}
	\end{bmatrix},
\]
where $K^{(1)}$ and $K^{(2)}$ are compact operators. Since $Q_0$ is invertible by definition, $I_{\Hs_1}-K^{(1)}$ is invertible and as a consequence
\[
	Q_0^{-1}=\begin{bmatrix}
	(I_{\Hs_1}-K^{(1)})^{-1}&\\
	K^{(2)}(I_{\Hs_1}-K^{(1)})^{-1}&I_{\Hs_0}
	\end{bmatrix}.
\]
The operator $P_0\in\Bs(\Hs_1\oplus\Hs_0)$ has the representation
\[
	P_0=(I_{\Hs}+K)H=\begin{bmatrix}
	(I_{\Hs_1}+V_1^*KV_1)H_1 &0\\
	V_0^*KV_1H_1&0 
	\end{bmatrix} 
\]
and from the block operator matrix representations of $Q_0^{-1}$ and of $P_0$ we obtain 
\[
	C_0=V_1^*Q_0^{-1}P_0V_1=(I_{\Hs_1}-K^{(1)})^{-1}(I_{\Hs_1}+V_1^*KV_1)H_1=(I_{\Hs_1}+K_1)H_1,
\]
where $K_1$ is compact. If $I_{\Hs_1}+V_1^*KV_1$ is invertible, then $I_{\Hs_1}+K_1$ is clearly invertible.
\end{proof}

\begin{lem}\label{2schurlem}
Assume the operators in \eqref{Adef2} satisfies $\ker H\subset \ker P_i$ for all $i=1,\hdots, k-1$ and $K=0$. Let $C$ denote the operator polynomial \eqref{Gdef2}. If  $\Hs_0=\emptyset$ or $I_{\Hs_0}+V_0^*\Kk V_0$ is invertible then, $I_{\Hs_1}+K_1$ is invertible.
\end{lem}
\begin{proof}
If $\Hs_0=\emptyset$, then $I_{\Hs_1}+K_1=Q_0^{-1}$, which is an invertible operator.

If $\Hs_0\neq\emptyset$, then from \eqref{Adef2}, \eqref{Cprops}, and the assumptions $\ker H\subset \ker P_i$, $i=1,\hdots, k-1$ it follows that
\[
	V_0^*Q_0V_0=V_0^*\left(I_{\Hs}+\Kk -\sum_{i=1}^{k}Q_iR_{k-i}\right)V_0=I_{\Hs_0}+V_0^*\Kk V_0,
\]
which is invertible by assumption. The operator
\[
	S:=V_1^*Q_0V_1-V_1^*Q_0V_0(V_0^*Q_0V_0)^{-1}V_0^*Q_0V_1,
\]
is a Schur complement of  $Q_0\in\Bs(\Hs_1\oplus\Hs_0)$ and $S^{-1}$ is bounded since $Q_0$ and $V_0^*Q_0V_0$ are invertible. Hence, $Q_0^{-1}$ has for some bounded operators $X$, $Y$, and $Z$ a block respresentaion in the form
\[
	Q_0^{-1}=\begin{bmatrix}
	S^{-1}&X\\
	Y&Z
	\end{bmatrix}=\begin{bmatrix}
	V_1^*Q_0^{-1}V_1&V_1^*Q_0^{-1}V_0\\
	V_0^*Q_0^{-1}V_1&V_0^*Q_0^{-1}V_0\\
	\end{bmatrix}.
\]
From the assumption $K=0$ and \eqref{eq:Lem32} follows then that $V_1^*Q_0^{-1}V_1=I_{\Hs_1}+K_1=S^{-1}$ is invertible. 
\end{proof}

\begin{lem}
Assume that $I_{\Hs_1}+K_1$ is invertible. Then
 $0\in \sigma_c(C)$ if and only if $\dim \Hs_1=\infty$.
\end{lem}
\begin{proof}
For $\dim \Hs_1<\infty$ follows $0\notin\sigma(C)$ trivially. Assume that $\dim \Hs_1=\infty$, then $0\in\sigma(C)\setminus\sigma_p(C)$ and we will therefore show that $0\notin \sigma_r(C)$. It can be seen that
\[
	C^*(0)=H_1^*(I_{\Hs_1}+K_1^*).
\]
Then $I_{\Hs_1}+K_1^*$ is invertible and since $H_1$ is normal with trivial kernel the operator $H_1^*$ has a trivial kernel. This implies that $C^*(0)$ has a trivial kernel and thus $0\notin \sigma_r(C)$.
\end{proof}
\begin{prop}\label{1propeq}
Let $P$ be defined as in \eqref{Adef2} and assume that $\ker H\subset \ker P_i$ for all $i=1,\hdots, k-1$. Moreover, let  $R$  denote the spectral divisor of order $k$ on some $\G\subset\C$, with $0\in \G$. Define $\Hs_0:=\ker H$, $\Hs_1:=\Hs\bot\Hs_0$, and $V_i$ as in \eqref{1partis}. Let $C=V_1^*RV_1$ denote the operator polynomial \eqref{Gdef2} and assume that $I_{\Hs_1}+K_1$ is invertible. 

If $\dim \Hs<\infty$ and $\Hs_0=\emptyset$ then $\sigma(P)\cap \G=\sigma(C)$.
If $\dim \Hs=\infty$ or $\Hs_0\neq \emptyset$ then $\sigma(P)\cap \G=\{0\}\cup\sigma(C)$. The Jordan chains of order $m$ corresponding to the eigenvalues in $\sigma(P)\cap \G\setminus\{0\}$ are given by \eqref{eigcorv}, where $\{u_i\}_{i=0}^{m-1}$ gives the corresponding Jordan chain to $C$. Moreover, the eigenvectors of $P(0)$ form a basis in $\Hs_0$ and the Jordan chains are of length $1$. 
\end{prop}

\begin{proof}
Since $R$ is a spectral divisor of $P$, Lemma \ref{1lemext} implies that $C$ is equivalent to $R$ after extension on $\Gamma\setminus\{0\}$.  For the point zero it follows by definition that any vector in $w\in \Hs_0$ is an eigenvector of $P(0)$. 
Hence, it is sufficient to show that $w$ is not an associated vector of $P(0)$ if $w\notin \Hs_0$. Assume $v\in\Hs_0$, $u\in\Hs_1$, and that $w:=u+v$ is an eigenvector of $P(0)$. Then
\begin{equation}\label{eq:P0eig}
	0=P(0)w=(I_\Hs+K)H u=0.
\end{equation}
From \eqref{eq:Lem32} follows
\[
	(I_{\Hs_1}+K_1)H_1u=V_1^*Q_0^{-1}(I_{\Hs_1}+K)V_1H_1u,
\]
where $V_1H_1u=Hu$ since $u\in \Hs_1$. Hence, the invertibility of $I_{\Hs_1}+K_1$, 
$\text{Ker}\, H_1=\{0\}$, and \eqref{eq:P0eig} imply that $u=0$ and hence $w\in\Hs_0$. We will now show that all Jordan chains are of length one. Assume $\{u_i+v_i\}_{i=0}^{1}$ is a Jordan chain of length two of $P(0)$, where $u_0,u_1\in\Hs_1$ and $v_0,v_1\in\Hs_2$. Then $u_0+v_0$ is an eigenvector and thus $u_0=0$. Hence, from definition of a Jordan chain
\[
	P(0)(u_1+v_1)+P'(0)v_0=(I_\Hs+K)H u_1+P_1v_0=0.
\]
If $k>1$ then $\ker H\subset\ker P_1$ and thus $P_1v_0=0$, which implies that $u_1+v_1$ is an eigenvector of $P(0)$. If $k=1$ then $P_1=I_\Hs+\Kk $ and \eqref{Cprops} implies
\begin{equation}\label{eq:assvec}
	(I_\Hs+K)H u_1+\left(Q_0+Q_1R_{0}\right)v_0=(I_\Hs+K)V_1H_1 u_1+Q_0v_0=0. 
\end{equation}
By multiplying \eqref{eq:assvec} with $V_1^*Q_0^{-1}$ we conclude that $(I_{\Hs_1}+K_1)H_1u_1=0$ and thus $u_1=0$, which yields that $u_1+v_1$ is an eigenvector of $P(0)$. 
\end{proof}

We are now ready to prove the main results of minimality and completeness of the set of eigenvectors and associated vectors, corresponding to a branch of eigenvalues of $P$ as well as accumulation of eigenvalues to the origin.

\begin{thm}\label{main}
Let $P$ be defined as \eqref{Adef2} with $k=1$ and let $-R_0$ denote the spectral root on some open set $\G\subset\C$ with $0\in \G$. Define $\Hs_0:=\ker H$, $\Hs_1:=\Hs\bot\Hs_0$, and $V_i$ as in \eqref{1partis}. Assume that the operator $I_{\Hs_1}+K_1$ defined in \eqref{eq:Lem32} is invertible and that the spectra of $H$ is located on a finite number of rays from the origin. 

Then, the set of eigenvectors and associated vectors  corresponding to the eigenvalues of $P$ in $\G$ are complete and minimal in $\Hs$. 
\end{thm}
\begin{proof}
Since $P_0$ is compact the minimality of the set of eigenvectors and associated vectors corresponding to non-zero eigenvalues of $P$ in $\G$ is the result of \cite[Theorem 22.13 a]{ASM}. 
From Corollary \ref{eigcor} none of the eigenvectors of $P(0)$ is an eigenvector or an associated vector of $P(\lambda)$ for $\lambda\neq0$. These vectors form a basis in $\Hs_0$, which implies that minimality extends to the set of eigenvectors and associated vectors corresponding to all eigenvalues in $\G$.

From Proposition \ref{1propeq} and that all vectors $v\in \Hs_0$ are eigenvectors of $P(0)$ it follows that the completeness result holds for $P$ if the set of eigenvectors and associated vector of $-C_0$ is complete in $\Hs_1$. From Lemma \ref{1lemext} it follows that $C_0=(I_{\Hs_1}+K_1)H_1$, $\ker H_1=\{0\}$ and the spectra of $H_1$ is located on a finite number of rays (since it coincides with the spectra of $H$ outside $0$). The completeness of the set of eigenvectors and associated vectors for $-C_0$ is then equivalent to the statement of \cite[Theorem 4.2]{ASM}.
\end{proof}

\begin{thm}\label{main2}
Let $P$ be defined as \eqref{Adef2} with $k\geq1$ and let  $R$ denote the spectral divisor of order $k$ on some $\G\subset\C$ with $0\in\G$. Define $\Hs_0:=\ker H$, $\Hs_1:=\Hs\bot\Hs_0$, and $V_i$ as in \eqref{1partis}. Assume that the operator $I_{\Hs_1}+K_1$ defined in \eqref{eq:Lem32}  is invertible  and that the spectra of $H$ is located on a finite number of rays from the origin.  

Then, the origin is an accumulation point of a branch of eigenvalues if $\dim \Hs_1=\infty$. Moreover, if $\dim \Hs_1<\infty$ and $\ker H\subset \ker P_i$ for $i=1,\hdots, k-1$ then the number of eigenvalues (repeated according to multiplicity) in $\G\setminus\{0\} $ is $k \dim \Hs_1$.
\end{thm}

\begin{proof}
 Lemma \ref{1lemsd} yields that it is sufficient to prove accumulation of eigenvalues to $0$ of $R$.
 Assume that $\dim \Hs_1=\infty$ and define for $z\in \C$ the compact operator
\begin{equation}\label{eq:Cz}
	\widehat{R}(z):=-R_0-\sum_{i=1}^{k-1}z^iR_i.
\end{equation}
Theorem \ref{main} implies that zero is an accumulation point of eigenvalues of $\widehat{R}(0)$. Let the sequence $\{\hat\lambda_n (0)\}_{n=1}^\infty$ denote the branch of accumulating eigenvalues, repeated according to its multiplicity and ordered non-increasingly in norm. Since $\widehat{R}(z)$ is compact it follows from \cite[Lemma 5, XI.9.5]{DUSC} that there exist eigenvalues $\{\hat\lambda_n (z)\}_{n=1}^\infty$ such that  $\hat\lambda_n (z)\rightarrow \hat\lambda_n (0)$, $z\rightarrow 0$.
  
  Define for $|z|>0$ the operator $\widetilde{R}(z):=\widehat{R}(z)z^{-k}$ and let $\omega_n(z):=\hat\lambda_n (z)z^{-k}$ denote its eigenvalues. From the triangle inequality
\begin{equation}\label{eq:tri}
	|\omega_n(z)|\geq \frac{|\hat\lambda_n (0)|-|\hat\lambda_n (z)-\hat\lambda_n (0)|}{|z|^k} 
\end{equation}
and the continuity of the eigenvalues follows that $|\omega_n(z)|>1$ for $z$ small enough. Define for fixed $|z|$ the function
\[
	f_n(t):=\frac{\hat\lambda_n  (|z|e^{it})}{|z|^ke^{ikt}}, \quad t\in [0,2\pi].
\]
From \eqref{eq:tri} follows that for $z$ small enough the closed curve $f_n$ rotates $k$ times around $1$. Moreover, $\|\widetilde{R}(z)\|\rightarrow 0$, $|z|\rightarrow \infty$ implies $\omega_n(z)\rightarrow 0$ as $|z|\rightarrow \infty$. Hence, from  continuity of $\omega_n$ and that the curve $f_n$ rotates around one for small $|z|$ it follows 
that $\omega_n(z)=1$ for some $z$. But then is $\hat\lambda_n(z)=z^k$ an eigenvalue of $\widehat{R}(z)$, which implies that $z$ is an eigenvalue of $R$.

We have shown that for any $n$ there exists a non-zero eigenvalue of $R$. Therefore, there exists an infinite sequence $\{\lambda_n\}_{n=1}^\infty\subset\sigma (R)$ of eigenvalues of finite multiplicities that accumulate to zero.

Now assume that $\dim \Hs_1<\infty$ and $\ker H\subset \ker P_i$, $i=1,\hdots, k-1$ then Proposition \ref{1propeq} yields that $R$ and $C$ have the same eigenvalues in $\C\setminus\{0\}$ with the same multiplicities. Furthermore, since $\dim \Hs_1<\infty$ and that $0$ is not an eigenvalue of $C$ it follows that $0\notin\sigma(C)$. Additionally, the eigenvalues coincide with the eigenvalues of the companion block linearization, \cite[Lemma 12.5]{ASM},  which is a linear operator of dimension $k \dim \Hs_1$.
\end{proof}

\section{Rational operator functions}
Let $A$ be a selfadjoint operator such that $0\notin \sigma_{ess}(A)$ and assume that  $\alc<\inf W(A)$ for some $\alc\in\R$. Take $N,M_1,\hdots, M_n\in\N$ and let $B_{i,j}\in\CL (\Hs)$ for $i\in\{1,\dots,N\}$,  $j\in\{1,\dots,M_i\}$, where $B_{i,j}$ are relatively compact to $A$. Further, let $S_i$, $i\in\{0,\hdots,M_0\}$  with $M_0\in\N \cup\{0\}$  denote operators relatively compact to $A$. Let $\tau_{i,j}$ be complex polynomials of degree less than $j$ and let $\delta_i\in \C$. Define the closed rational operator function $T:\C\setminus\{\delta_1,\hdots,\delta_{N}\}\rightarrow \Ls(\Hs)$ as
\begin{equation}\label{eq:Tdef}
	T(\omega):=A+\sum_{i=0}^{M_0}\omega^iS_i+\sum_{i=1}^{N}\sum_{j=1}^{M_i}\frac{\tau_{i,j}(\omega)}{(\omega-\delta_i)^j}B_{i,j}, 
\end{equation}
with $\dom T(\omega)=\dom A$ and $\rho(T)\neq\emptyset$. Without loss of generality, we assume that $\delta_i$ are distinct and $\tau_{i,j}(\delta_i)\neq0$.

The claims in Lemma \ref{lem:disc} and in Lemma \ref{3simob} are standard results for Fredholm-valued operator functions \cite[\textsection 20]{ASM} formulated in our setting. For convenience of the reader we provide short proofs.

\begin{lem}\label{lem:disc}
The spectra of \eqref{eq:Tdef} consists of discrete eigenvalues with $\delta_1,\hdots,\delta_N$ and $\infty$ as their only possible accumulation points. 
\end{lem}
\begin{proof}
Since $0\notin \sigma_{ess}(A)$ there is a selfadjoint finite rank operator $K$ such that  $0\notin\sigma(A+K)$. Clearly $B_{i,j}$ and $S_i$ are relatively compact to $A+K$. Since $\rho (T)\neq\emptyset$ and $\Ran ((A+K)^{-1})=\dom A$, the result follows directly from \cite[Lemma 20.1, Lemma 20.2]{ASM}. 
\end{proof}

\begin{lem}\label{3simob}
Let $T$ denote the closed operator function \eqref{eq:Tdef} and define the bounded operator function
\begin{equation}\label{transb}
		\widehat{P}(\omega)=(A-\alc)^{-\frac{1}{2}}T(\omega)(A-\alc)^{-\frac{1}{2}}\prod_{i=1}^{N}(\omega-\delta_i)^{M_i},
\end{equation}
which can be extended to an operator polynomial $\widehat{P}:\C\rightarrow\Bs(\Hs)$. Then 
\[
	\sigma(\widehat{P})\setminus\{\delta_1,\hdots,\delta_{N}\}=\sigma_{disc}(\widehat{P})\setminus\{\delta_1,\hdots,\delta_{N}\}=\sigma_{disc}(T)=\sigma(T),
\]
and
\[
	 W(\widehat{P})\setminus\{\delta_1,\hdots \delta_{N}\}=W(T).
\]
Further, $\{u_i\}_{i=0}^{j-1}$ is a Jordan chain of length $j$ for  $\omega\in\sigma_{disc}(\widehat{P})$ if and only if  $\{(A-\alc)^{-\frac{1}{2}}u_i\}_{i=0}^{j-1}$ is a Jordan chain of length $j$ for  $\omega\in\sigma_{disc}(T)$. Moreover, the zeros of $\ip{\widehat{P}(\omega)u}{u}$ and of $\ip{T(\omega)(A-\alc)^{-\frac{1}{2}}u}{(A-\alc)^{-\frac{1}{2}}u}$ coincide for all $u\in\dom(A^\frac{1}{2})$.
\end{lem}
\begin{proof}
Define $\widehat{T}(\omega):=(A-\alc)^{-\frac{1}{2}}T(\omega)(A-\alc)^{-\frac{1}{2}}$. Lemma \ref{lem:disc} yields that the spectrum of $T$ is discrete. Hence, since $\widehat{T}(\omega)$ is a compact perturbation of $T(\omega)(A-\alc)^{-1}$ it follows from \cite[Lemma 20.1]{ASM} that $\widehat{T}$ has also has discrete spectrum in $\C\setminus\{\delta_1,\hdots,\delta_{N}\}$. The result then follows by straight forward computations.
\end{proof} 

\begin{lem}\label{ratpolmult}
Let $\widehat{P}$ denote the operator polynomial \eqref{transb} and define $\widehat{B}_{i,j}:=(A-\alc)^{-\frac{1}{2}}B_{i,j}(A-\alc)^{-\frac{1}{2}}$ for $i\in\{1,\hdots,N\}$, $j\in \{1,\hdots,M_i\}$. Set $\lambda_l:=\omega-\delta_l$ and 
define the operator polynomial 
\begin{equation}\label{penmult}
	P^{(l)}(\lambda_l):=\dfrac{\widehat{P}(\omega)}{\prod_{i\in\{1,\hdots N\}\setminus\{l\}}(\delta_l-\delta_i)^{M_i}}.
\end{equation}
Then $P^{(l)}$ satisfies the conditions 
\begin{equation}\label{Ai02mult}
	\begin{array}{l}
	P^{(l)}_i=\tau_{l,M_l-i}(\delta_l)\widehat{B}_{l,M_l-i}+\sum_{j=0}^{i-1}x_{i,j}\widehat{B}_{l,M_l-j}, \quad \text{for }i<M_l,\\
	P^{(l)}_{M_l}=I+(A-\alc)^{-\frac{1}{2}}(\alc+X)(A-\alc)^{-\frac{1}{2}},
	\end{array}
\end{equation}
where $x_{i,j}\in \C$ and $X$ is a linear combination of the operators $\widehat{B}_{i,j}$ and $(A-\alc)^{-\frac{1}{2}}S_l(A-\alc)^{-\frac{1}{2}}$, $l\in\{0,\hdots,M_0\}$. 
\end{lem}
\begin{proof}
This can be seen from straight forward computations. 
\end{proof}
\begin{cor}\label{ratpolcormult}
Let $T$ denote \eqref{eq:Tdef} and let $P^{(l)}$ denote the operator polynomial \eqref{penmult}. Then 
\[
	\sigma(P^{(l)})\setminus\{\delta_1-\delta_l,\hdots,\delta_N-\delta_l\}=\{\omega-\delta_l:\omega\in\sigma(T)\},
\]
and 
\[
	\overline{W(P^{(l)})}\setminus\{\delta_1-\delta_l,\hdots,\delta_N-\delta_l\}=\{\omega-\delta_l:\omega\in\overline{W(T)}\}.
\]
Further, $\{u_i\}_{i=0}^{j-1}$ is a Jordan chain of length $j$ for $\omega-\delta_l\in\sigma_{disc}(P^{(l)})$ if and only if  $\{(A-\alc)^{-\frac{1}{2}}u_i\}_{i=0}^{j-1}$ is a Jordan chain of length $j$ for $\omega\in\sigma_{disc}(T)$. Moreover, the roots of $\ip{P^{(l)}(\omega-\delta_j)u}{u}$ and the zeros of $\ip{T(\omega)(A-\alc)^{-\frac{1}{2}}u}{(A-\alc)^{-\frac{1}{2}}u}$ coincide for all $u\in\dom(A^\frac{1}{2})$.
\end{cor}
\begin{proof}
Follows directly from Lemma \ref{3simob} and Lemma \ref{ratpolmult}.
\end{proof}

\begin{lem}\label{mobrootmult}
Let $T$ be defined as \eqref{eq:Tdef}.   Assume that for some $j\in\{1,\hdots,N\}$ there is a simply connected open bounded  set $\Gh_l\subset\C$ such that $\delta_l\in \Gh_l$, $\delta_j\notin\Gh_l$ for $j\neq l$ and  $\partial\Gh_l\cap \overline{W(T)}=\emptyset$. Additionally assume that $\Gh_l$ is a disk if $m_l>1$.  Then there is a spectral divisor of order $m_l$ of the operator polynomial \eqref{penmult} on $\G:=\{\omega-\delta_l:\omega\in\Gh_l\}$ and $0\in \G$.
\end{lem}
\begin{proof}  
By definition $\G$ is bounded, $0\in \G$, and $\delta_j-\delta_l\notin \G$  for $j\in\{1,\hdots,N\}\setminus\{l\}$. Then Corollary \ref{ratpolcormult} yields that $\partial \G\cap \overline{W(P^{(l)})}=\emptyset$, which implies that
\[
	\inf_{u\in\Hs}|\ip{P^{(l)}(\lambda) u}{u}|>0\quad\text{for }\lambda\in\partial \G.
\]
From the definition of $P^{(l)}$ it follows that
\[
	\ip{P^{(l)}(\lambda)u}{u}=\lambda^{M_l}\prod_{i\in\{1,\hdots,N\}\setminus\{l\}}\frac{(\lambda+\delta_l-\delta_i)^{M_i}}{(\delta_l-\delta_i)^{M_i}}
	+\upsilon_u(\lambda),
\]
where $\upsilon_u(\lambda)$ is a polynomial whose coefficients depend on $(A-\alc)^{-\frac{1}{2}}u$. Since $A$ is unbounded there is a sequence $\{u_i\}_{i=1}^\infty\in\Hs$ such that $(A-\alc)^{-\frac{1}{2}}u_i\rightarrow 0$. Hence, the finite roots of $\lim_{i\rightarrow \infty}\ip{P_{\pm}(\lambda)u_i}{u_i}$ are $\delta_l-\delta_j$ with multiplicity $M_j$ for all $j\in\{1,\hdots,N\}$. Continuity then yields that there are exactly $M_l$ roots in $\G$. Hence, \cite[Theorem 26.19 and Theorem 26.13]{ASM} yield that we have a spectral divisor for $x_l=1$ and for $x_l>1$, respectively.
  \end{proof}

\begin{thm}\label{prop:mult}
Let $T$ denote the operator function \eqref{eq:Tdef} and assume that $B_{l,M_l}$ for some $l\in\{1,\hdots,N\}$ is selfadjoint with $(A-\alc)^{-\frac{1}{2}}B_{l,M_l}(A-\alc)^{-\frac{1}{2}}\in S_p(\Hs)$.

\begin{itemize}
\item[{\rm (i)}]  Assume that $M_l=1$ and that $A$ is sufficiently large. Then there is a branch of eigenvalues accumulating at $\delta_l$ if and only if $\Hs\bot\ker B_{l,M_l}$ is infinite dimensional.  
Let $\Hs_0:=\ker (A-\alc)^{-\frac{1}{2}}B_{l,M_l}(A-\alc)^{-\frac{1}{2}}$ and let $\Hs_1:=\Hs\bot\Hs_0$ denote the Hilbert spaces with the inner products $\ip{\cdot}{\cdot}_{\Hs_i}=\ip{V_i\cdot}{V_i\cdot}$, where $V_i$ is defined in \eqref{1partis}. Assume that $\{u_j\}$ is the set of eigenvectors and associated vectors corresponding to the branch of  eigenvalues accumulating at $\delta_l$. Then the set $\{V_1^*(A-\alc)^\frac{1}{2}u_j\}$ is complete and minimal in $\Hs_1$.
 
\item[{\rm (ii)}] 
 Assume that $M_l>1$ and $A$ is sufficiently large. Then if $\Hs\bot\ker B_{l,M_l}$ is infinite dimensional there is a branch of eigenvalues accumulating at $\delta_l$.  If $\Hs\bot\ker B_{l,M_l}$ is not infinite dimensional and  $\ker B_{l,M_l}\subset \ker B_{l,i}$ for  $i=1,\hdots, M_l-1$, then there are $M_l \dim \Hs\bot\ker B_{l,M_l}$ eigenvalues, counting multiplicity, corresponding to the branch of eigenvalues.
 \end{itemize}
 \end{thm}
\begin{proof}
Corollary \ref{ratpolcormult} shows that the spectrum of $T$ is determined by the spectrum of $P^{(l)}$. Similarly as in the proof of Lemma \ref{mobrootmult} it follows that if $A$ is large enough then there is a  component of $\overline{W(P^{(l)})}$ that is contained in some open bounded set $\G\subset \C$ with $0\in\G$ and there are $M_l$ roots of $\ip{P^{(l)}(\lambda)u}{u}$ for all $\lambda\in\G$, where $\G$ is a disk if $M_l>1$. Then similarly as in Lemma \ref{mobrootmult} it follows that $P^{(l)}$ has a spectral divisor of order $M_l$ in  $\G$. From Lemma \ref{ratpolmult}  we obtain
\[
	P^{(l)}_0=\tau_{l,M_l}(\delta_l)(A-\alc)^{-\frac{1}{2}}B_{l,M_l}(A-\alc)^{-\frac{1}{2}}, 
\]
which is a normal operator with spectrum on a ray from the origin and $P^{(l)}_i$, $i=0,\hdots,M_l-1$ are compact. Moreover, if $\ker B_{l,M_l}\subset \ker B_{l,i}$, $i=1,\hdots, M_l-1$ then  $\ker P^{(l)}_0\subset \ker P^{(l)}_i$ for  $i=1,\hdots, M_l-1$. If $A$ is unbounded, then $(A-\alc)^{-1}$ is compact, otherwise since $A$ is sufficiently large, $\alc$ can be chosen to $0$. Hence, in both cases it follows that $P^{(l)}_{M_l}$ can be written in the form $I+\widetilde{K}$, where $\widetilde{K}$ is compact. Hence, $P^{(l)}$ has the form of \eqref{Adef2} with $k=M_l$ and $K=0$.  Let $H:=P_0^{(j)}$, $\Hs_0:=\ker H$, $\Hs_1:=\Hs\bot\Hs_0$, and let $V_i$ denote the operator \eqref{1partis}. From Lemma \ref{ratpolmult}
and the properties of the spectral divisor given in  \cite[Theorem 22.11 and Theorem 24.2]{ASM} it follows that for sufficiently large $A$, the operator $I_{\Hs_1}+K_1$ defined in \eqref{eq:Lem32} is a small perturbation of $I_{\Hs_1}$ and thus invertible.
The result then follows from Theorem \ref{main} and Theorem \ref{main2}.
\end{proof}

Now accumulation of complex eigenvalues to $\infty$ is studied and therefore we introduce in Lemma \ref{ratpolinf}  an auxiliary operator polynomial.
\begin{lem}\label{ratpolinf}
Let $\widehat{P}$ denote the operator function \eqref{transb} and set $\widehat{S}_i:=(A-\alc)^{-\frac{1}{2}}S_i(A-\alc)^{-\frac{1}{2}}$ for $i\in\{0,\hdots,M_0\}$. Define the operator function
\[
	\widetilde{P}(\lambda):=\lambda^n\widehat{P}(\lambda^{-1}), \quad n=\sum_{i=0}^N M_i,
\]
which can be extended to an operator polynomial $\widetilde{P}:\C\rightarrow\Bs(\Hs)$. The coefficients in the polynomial are 
\[
	\begin{array}{l}
	\widetilde{P}_i=\widehat{S}_{M_0-i}+\sum_{j=0}^{i-1}x_{i,j}\widehat{S}_{M_0-j}, \quad \text{for }i<M_0,\\
	\widetilde{P}_{M_0}=I+(A-\alc)^{-\frac{1}{2}}(\alc+X)(A-\alc)^{-\frac{1}{2}},
	\end{array}
\]
where $x_{i,j}\in \C$ and $X$ is a linear combination of the operators $\widehat{S}_i$.
\end{lem}
\begin{proof}
Follows by straight forward computations.
\end{proof}

\begin{thm}\label{prop:inf}
Let $T$ denote the operator function \eqref{eq:Tdef}, assume that $M_0>0$ and  $S_{M_0}$ is selfadjoint with $(A-\alc)^{-\frac{1}{2}}S_{M_0}(A-\alc)^{-\frac{1}{2}}\in S_p(\Hs)$.

\begin{itemize}
\item[{\rm (i)}]  Assume that $M_0=1$ and $A$ is sufficiently large. Then there is a branch of eigenvalues accumulating at complex $\infty$ if and only if $\Hs\bot\ker S_{M_0}$ is infinite dimensional. 
Let $\Hs_0:=\ker (A-\alc)^{-\frac{1}{2}}S_{M_0}(A-\alc)^{-\frac{1}{2}}$ and let $\Hs_1:=\Hs\bot\Hs_0$ denote the Hilbert spaces with the inner products $\ip{\cdot}{\cdot}_{\Hs_i}=\ip{V_i\cdot}{V_i\cdot}$, where $V_i$ is defined in \eqref{1partis}. Assume that $\{u_j\}$ is the set of eigenvectors and associated vectors corresponding to the branch of  eigenvalues accumulating at $\infty$. Then the set $\{V_1^*(A-\alc)^\frac{1}{2}u_j\}$ is complete and minimal in $\Hs_1$.
 
\item[{\rm (i)}] 
Assume that $M_0>1$ and $A$ is sufficiently large. Then if $\Hs\bot\ker S_{M_0}$ is infinite dimensional there is a branch of eigenvalues accumulating at complex $\infty$.  If $\Hs\bot\ker S_{M_0}$ is  finite dimensional and  $\ker S_{M_0}\subset \ker S_{i}$ for  $i=1,\hdots, M_0-1$, then there are $M_l \dim \Hs\bot\ker S_{M_0}$ eigenvalues, counting multiplicity, corresponding to the branch of eigenvalues. 
\end{itemize}
 \end{thm}
 \begin{proof}
The proof is completely analogous to that of Theorem \ref{prop:mult}, with the auxiliary operator polynomial in Lemma \ref{ratpolinf}.
 \end{proof}
\section{A class of rational operator functions}
Let $A$ be a selfadjoint operator with $(A-\alc)^{-1}\in S_p(\Hs)$ for some $p\geq1$ and assume that $\inf W(A)>\alc$,  for some $\alc\in\R$. Further, let $B_j$,  $j=1,\hdots,N$ be bounded selfadjoint operators with $\inf W(B_j)\geq0$.
Define the unbounded rational operator function $T:\domain\rightarrow \Ls(\Hs)$ as
 \begin{equation}\label{eq:Tdefex}
	T(\omega):=A-\omega^2-\sum_{j=1}^{N}\frac{\omega^2}{c_j-id_j\omega-\omega^2}B_j, \quad \domain:=\{\omega\in\C\,:\ \prod_{j=1}^{N}(c_j-id_j\omega-\omega^2)\neq 0\},\end{equation}
with $\dom T(\omega)=\dom A$ and $c_j\geq0$, $d_j\geq 0$, for $j\in\{1,\hdots,N\}$. Since $A$ is bounded from below and the operators $B_j$ are bounded, the operator $T(i\mu)$ is invertible for $\mu\in\R$ sufficiently large. Hence, the operator function \eqref{eq:Tdefex} can be written in form \eqref{eq:Tdef} and Theorem \ref{prop:mult} implies that under certain conditions there are branches of eigenvalues accumulating at all the poles of \eqref{eq:Tdefex} for $A$ "sufficiently large". Our aim is to determine a lower bound on $A$ that ensure that there exist components of the numerical range of $T$ such that Lemma \ref{mobrootmult} is applicable and yields a spectral divisor. In Subsection \ref{sec:oneterm}, we use the enclosure of the numerical range introduced in \cite{ATEN} to derive explicit sufficient conditions for accumulation of eigenvalues for the case $N=1$. This approach can in principle be used for $N>1$ but the derivations would be extremely technical.

\subsection{One rational term}\label{sec:oneterm}
Consider the rational operator function defined in \eqref{eq:Tdefex} for $N=1$.
Let $\alpha_0:=\inf W(A)$,  $[\beta_0,\beta_1]:=\overline{W(B)}$ and introduce the constants $\theta:=\sqrt{c-d^2/4}$, $\delta_\pm=\pm\theta-id/2$. Then the operator function is  $T:\C\setminus\{\delta_+,\delta_-\}\rightarrow \Ls(\Hs)$,
\begin{equation}\label{tdef}
	T(\omega):=A-\omega^2-\frac{\omega^2}{c-id\omega-\omega^2}B.
\end{equation}
The case $d=0$ was studied in \cite{MR3543766} and we will therefore assume that $d>0$. The key in our derivation of  sufficient conditions for accumulation of eigenvalues is to prove the existence of a spectral divisor, which depend on properties of the numerical range. The enclosure of the numerical range introduced in \cite{ATEN} is optimal given only the numerical ranges $W(A)$, W(B) and it is of vital importance for the proofs of our results. Define
\begin{equation}\label{3pabdef}
p_{\alpha,\beta}(\omega):=(\alpha-\omega^2)(c-id\omega-\omega^2)-\omega^2\beta
\end{equation}
and order the roots 
\begin{equation}\label{2roots}
r_n:\R\times\R\rightarrow\C,\ n=1,\dots,4,
\end{equation}
of $p_{\alpha,\beta}$ such that they are continuous functions in $(\alpha,\beta)$. The numerical range of $T$ is then
\[
	W (T):=\bigcup_{n=1}^4\left\{r_n(\ip{Au}{u},\ip{Bu}{u}): u\in \dom A,\|u\|=1\right\}.
\]
Since $(\ip{Au}{u},\ip{Bu}{u})\in \Omega:=\overline{W(A)}\times\overline{W(B)}$ for all $u\in \dom A$ with $\|u\|=1$ we obtain an enclosure of $\overline{W(T)}\setminus\{\delta_+,\delta_-\}$ as 
\begin{equation}\label{3endef}
	W_\Omega(T):=\bigcup_{n=1}^4\left\{r_n(\alpha,\beta):(\alpha,\beta)\in \Omega\right\}.
\end{equation}
Note that in \cite{ATEN} the set $W_\Omega(T)$ is defined for the Riemann sphere $\eC$ instead of $\C$. The roots of $p_{\alpha,\beta}$ have the poles $\delta_+$, $\delta_-$, and $\pm\infty$ as limits when $\alpha\rightarrow\infty$. Hence, the set $W_\Omega(T)$ consists of four disjoint components if $A$ is unbounded and the lower bound $\alpha_0>\alpha'$ is large enough. Below, we determine a lower bound on $A$ that ensures that there exists a bounded component of the numerical range of $T$ that contain a pole. The cases $d\neq2\sqrt{c}$ and $d=2\sqrt{c}$ are studied separately, since we in the case $d=2\sqrt{c}$ have a second order pole in $T$. 

\subsection{The case $d\neq2\sqrt{c}$}
In this section we present sufficient conditions on $T$  that assure that Theorem \ref{main} is applicable, and thus guarantee the existence of an infinite sequence of eigenvalues accumulating to $\delta_\pm$. 
\begin{lem}\label{ratpol}
Assume $d\neq2\sqrt{c}$ and let $\widehat{P}$ denote the operator polynomial \eqref{transb}. Define the polynomials $P^\pm$ for the poles $\delta_\pm$ as in \eqref{penmult}. Then
\begin{equation}\label{pen}
P^\pm(\lambda):=\sum_{j=0}^4 \lambda^jP_j^\pm,\quad \lambda\in \C,
\end{equation}
where $H^\pm:=\pm\frac{\delta_\pm^2}{2\theta} (A-\alc)^{-\frac{1}{2}}B(A-\alc)^{-\frac{1}{2}}$ and
\begin{equation}\label{Ai02}
	\begin{array}{l}
	P_0^\pm:=H^\pm,\\
	P_1^\pm:=I+(\alc - \delta_\pm^2)(A-\alc)^{-1}+ \dfrac{2}{\delta_\pm }H^\pm,\\
	P_2^\pm:=\pm\dfrac{1}{2\theta}\left(I+(\alc-  \delta_\pm^2)(A-\alc)^{-1}\right) -2  \delta_\pm(A-\alc)^{-1} +\dfrac{1}{\delta_\pm^2}H^\pm,\\
	P_3^\pm:=\left(\pm \dfrac{id}{2\theta}-2\right)(A-\alc)^{-1},\\
	P_4^\pm:=\mp \dfrac{1}{2\theta}(A-\alc)^{-1}.
	\end{array}
\end{equation}
\end{lem}
\begin{proof}
The representation follows directly from definition \eqref{penmult}.
\end{proof}

If $c=0$ only the pole $\delta_-=-id$ is of interest since in this case $\delta_+=0$ is a removable singularity. 
\begin{lem}\label{3invlem}
Assume $d<2\sqrt{c}$ and let $P^\pm$ and $H^\pm$ be defined as  \eqref{pen}. Set $\Hs_0:=\ker H^\pm$ and define $V_0$ as in \eqref{1partis}. Then $V_0^*P_1^{\pm}V_0$ is invertible.
\end{lem}
\begin{proof}From \eqref{Ai02} it follows that
\[
	V_0^*P_1^{\pm}V_0=I_{\Hs_0}+\left(\alc-c+\frac{d^2}{2}\pm id\sqrt{c-\frac{d^2}{4}}\right)V_0^*(A-\alc)^{-1}V_0.
\]
Since $P_1^{\pm}=I+ K_\pm$, where $K_\pm$ is compact, it follows that if  $0\in \sigma(V_0^*P_1^{\pm}V_0)$ then zero is an eigenvalue. Assume $0\in \sigma(V_0^*P_1^{\pm}V_0)$  and that $u$ is a corresponding eigenvector, then 
 \[
	\Im\ip{V_0^*P_1^{\pm}V_0u}{u}=\pm d\sqrt{c-\frac{d^2}{4}}\ip{V_0^*(A-\alc)^{-1}V_0u}{u}=0.
\] 
Since $d\sqrt{c-d^2/4}>0$, it follows that $\ip{V_0^*(A-\alc)^{-1}V_0u}{u}=0$ and consequently $\ip{V_0^*P_1^{\pm}V_0u}{u}=1$, which contradicts the assumptions.
\end{proof}

\begin{lem}\label{3invlemi}
Assume $d>2\sqrt{c}$ and let $P^\pm$, $H^\pm$ be defined as in Lemma \ref{ratpol}. Set $\Hs_0:=\ker H_\pm$ and define $V_0$ as in \eqref{1partis}. If $\ac>\delta_\pm^2$ then $V_0^*P_1^{\pm}V_0$ is invertible. 
\end{lem}
\begin{proof}
From \eqref{Ai02} it follows that $V_0^*P_1^{\pm}V_0=I_{\Hs_0}+(\alc-\delta_\pm^2)V_0^*(A-\alc)^{-1}V_0$. If $\alc\geq\delta_\pm$ then $V_0^*P_1^{\pm}V_0$ is clearly invertible. For $\alc<\delta_\pm$ the inequality $\|(A-\alc)^{-1}\|\leq (\ac-\alc)^{-1}<(\delta_\pm-\alc)^{-1}$ holds. Hence, $\|(\alc-\delta_\pm^2)V_0^*(A-\alc)^{-1}V_0\|<1$ and thus $V_0^*P_1^{\pm}V_0$ is invertible.
\end{proof}

To be able to utilize Theorem \ref{main}, we need to find a spectral root (a spectral divisor of order $1$) of $P_{\pm}$. 

\begin{lem}\label{mobroot}
Let $T$ denote the operator function \eqref{tdef} and set  $\Omega:=\overline{W(A)}\times \overline{W(B)}$. Assume that $d\neq2\sqrt{c}$  and that there is a simply connected bounded open set $\Gh_\pm\subset\C$ such that $\delta_\pm\in \Gh_\pm$, $\delta_\mp\notin\Gh_\pm$, and  $\partial\Gh_\pm\cap W_\Omega(T)=\emptyset$.  Then there is a spectral root of the operator polynomial \eqref{pen} on $\G:=\{\omega-\delta_\pm:\omega\in\Gh_\pm\}$ with $0\in \G$.
\end{lem}
\begin{proof}
Follows from Lemma \ref{mobrootmult} and the property $W_\Omega(T)\supset \overline{W(T)}$.
\end{proof}

The points in the set $W_\Omega(T)$ defined in \eqref{3endef} depend continuously on $\alpha$ and we can therefore make the following definition:
\begin{defn}\label{3mergdef} 
Two components of $W_\Omega(T)$ are said to \emph{merge} if $W_{[\alpha,\infty)\times  \overline{W(B)}}(T)$ has more components than $W_{[\alc,\infty)\times \overline{W(B)}}(T)$ for all $\alpha>\alc$. Assume that two components of $W_\Omega(T)$ merge and let $\oc\in W_{[\alc,\infty)\times \overline{W(B)}}(T)$ be a point such that the distance to each of the merging components of $W_{[\alpha,\infty)\times \overline{W(B)}}(T)$ for $\alpha>\alc$ goes to zero as $\alpha\rightarrow\alc$. The point $\oc$ is then called a  \emph{critical point} of $W_\Omega(T)$.
\end{defn} 

The goal is to find the largest $\alc$ such that the component containing $\delta_\pm$ merges with some other component. Then the conditions of Lemma \ref{mobroot} hold for $\alc<\ac$. An important property of the roots $r_n$ is that if $r_n(\alc,\bec)\neq r_m(\alc,\bec)$ then $r_n(\alc,\bec)\neq r_m(\alc,\beta)$ for all $\beta\in \overline{W(B)}$ \cite[Lemma 2.8]{ATEN}. This implies the following property at a critical point: if $\oc$ is a critical point then either there exists a $\bec\in \overline{W(B)}$ such that $\oc$ is a root of multiplicity at least two of $p_{\alc,\bec}$ or $r_n(\alc,\bec)=\oc$  for some  $n\in\{1,2,3,4\}$ and $r_m(\alpha,\beta)=\oc$, $n\neq m$
for some $\alpha>\alc$ and $\beta\in\overline{W(B)}$. In the following we will utilize this property in the analysis of the critical points. 

\begin{lem}\label{3lemdou}
Let $T$ denote the operator function \eqref{tdef} and let $W_{\Omega}(T)$ denote the set \eqref{3endef}. Assume that $W_\Omega(T)$ has a critical point $\oc\in \C\setminus i\R$.  Then
\[
	\alc=c>\frac{d^2}{16},\quad \frac{d^2}{4}\in W(B),\quad \oc=\mp\sqrt{c-\frac{d^2}{16}}-i\frac{d}{4}.
\]
\end{lem}
\begin{proof}
From \cite[Proposition 2.6]{ATEN} it follows that for each $\omega\in\C\setminus \{i\R,\delta_+,\delta_-\}$, $\omega\in W_\Omega(T)$ if and only if $(\alpha,\beta)\in \Omega$, where $(\alpha,\beta)$ are uniquely given. Hence, if two components of $W_\Omega(T)$  has a critical point $\oc\in \C\setminus i\R$  
then $\oc$ must be a double root of $p_{\alc,\bec}$ and the result follows from \cite[Lemma 2.8]{ATEN}.
\end{proof}

If $\alc\geq0$ and $d<2\sqrt{c}$ then $W_\Omega(T)\cap i\R=\emptyset$, which combined with Lemma \ref{3lemdou} yields that $W_\Omega(T)$ has four disjoint component if $\overline{W(A)}\subset (c,\infty)$ or  $d^2/4\notin \overline{W(B)}$. However, critical points $\oc\in i\R$ are in general much more technical to determine. In the following, we  provide a machinery for finding a proper $\alc$ with a critical point on $i\R$. 

\begin{lem}\label{3corhighlow}
Let $T$ be defined as in \eqref{tdef} and assume that $i\mu\in W_\Omega(T)$ for some $\mu\in\R$. Assume that $\alc$ is the largest value such that $i\mu$ is a root of $p_{\alpha',\bec}$ for some $\bec\in \overline{W(B)}$. Then $\bec=\beta_0$ if $c+d\mu+\mu^2\geq0$ and $\bec=\beta_1$ if $c+d\mu+\mu^2<0$.
\end{lem}
\begin{proof}
Follows directly from the definition of $p_{\alpha,\beta}$. 
\end{proof}

\begin{lem}\label{lemtripleroot}
Let $p_{\alpha,\beta}$ be defined as in \eqref{3pabdef}. Then $i\mu$ with $\mu\in \R$ is a root of at least order $3$ of $p_{\alpha,\beta}$ if and only if $\alpha=\alpha_\pm$, $\beta=\beta_\pm$, and $\mu=\mu_\pm$, where
\begin{equation}\label{3en3root}
\begin{array}{l}
	\alpha_\pm=\dfrac{-4096c^3 + 768c^2d^2 + 6cd^4 + d^6 \pm\sqrt{(4c - d^2)(16c - d^2)}(16c - d^2)(32c + d^2)}{54d^4},\\
	\beta_\pm=\dfrac{4(-4c + d^2)\left((32 c+d^2)(-4c+d^2) \pm\sqrt{(4c - d^2)(16c - d^2)}(16c  -  d^2)\right)}{27d^4},\\
	\mu_\pm=\dfrac{-8c-d^2 \pm\sqrt{(4c - d^2)(16c  -  d^2)}}{6d}.
	\end{array}
\end{equation}
\end{lem}
\begin{proof}
The ansatz that $i\mu$ is a triple root of $p_{\alpha,\beta}$ gives a system of equations in $\alpha$, $\beta$, and $\mu$ with the stated solutions. 
\end{proof} 
\begin{cor}\label{cortripleroot}
The constants $\beta_\pm$ defined in Lemma \ref{lemtripleroot} satisfy the following equalities,
\[
\begin{array}{l c l}
		\beta_-=\dfrac{d^2}{4}&\text{if and only if}& c=\dfrac{d^2}{16} \text{ or } c=\dfrac{5d^2}{16},\\
	\beta_+=\dfrac{d^2}{4}&\text{if and only if}&  c=\dfrac{d^2}{16},\\
	\beta_\pm=0&\text{if and only if} & c=\dfrac{d^2}{4},\\
	\end{array}
\]
and the inequalities
\[
\begin{array}{l c l}
	\beta_->\dfrac{d^2}{4} &\text{ if }&  c>\dfrac{5d^2}{16},\\
	0<\beta_-<\dfrac{d^2}{4} &\text{ if }&  \dfrac{d^2}{4}<c<\dfrac{5d^2}{16},\\
	\beta_+<0 &\text{ if }& c>\dfrac{d^2}{4}.\\
	\end{array}
\]
\end{cor}
\begin{proof}
The statements follow from straightforward calculations.
\end{proof}

\begin{lem}\label{3lemdoui}
Let $T$ denote the operator function \eqref{tdef} and let $W_{\Omega}(T)$ denote the set \eqref{3endef}. Assume that $W_\Omega(T)$ has a critical point  $i\mu \in i\R$. Then one of the following properties hold:
\begin{itemize}
\item[{\rm (i)}]   $i\mu$ is a root of order at least $2$ of $p_{\alc,\bec}$,  where $\bec=\beta_0$ if  $c+d\mu+\mu^2\geq0$, and $\bec=\beta_1$ if  $c+d\mu+\mu^2<0$.
\item[{\rm (ii)}] $\alc=\alpha_-$ and $i\mu=i\mu_-\notin \{\delta_+,\delta_-\}$ is a root of order at least $3$ of $p_{\alpha',\beta_-}$, where $\alpha_-$, $\beta_-$, and $\mu_-$ are defined in \eqref{3en3root}.
\end{itemize}
\end{lem}
\begin{proof}
Let $r_n(\alpha,\beta)$ denote the roots of $p_{\alpha,\beta}$ defined as in \eqref{2roots}.  First assume that $\mu=0$, then $p_{\alc,\bec}(i\mu)=\alc c=0$, and $c+d\mu+\mu^2\geq0$. If  $\alc=0$, then $i\mu$ is a double root of $p_{\alc,\beta_0}$. If $c=0$ then for some $n\in\{1,2,3,4\}$ it holds that $r_n(\alpha,\beta)=0$ for all $\alpha\in\R$ and $\beta\in \R$. Since $0$ is a critical point it follows that for some $m\neq n$  and $\bec\in \overline{W(B)}$ we have $r_m(\alc,\bec)=0$. Hence $0$ is a root of order two of $p_{\alc,\bec}$ and Proposition \ref{3corhighlow} then yields that $\bec=\beta_0$.
 Now assume that $c+d\mu+\mu^2=0$ for some $\mu\neq 0$. Then $p_{\alc,\bec}(i\mu)=\mu^2\bec=0$ implies $\bec=0$ and the assumption $B\geq0$ yields $\bec=\beta_0=0$. But then $i\mu$ is a root of $p_{\alpha,\bec}$ for all $\alpha$, and thus it follows that $i\mu$ is a double root of $p_{\alc,\bec}$. In the remaining part of the proof we assume $\mu\notin \{0,\delta_+,\delta_-\}$. 

Suppose that the critical point $i\mu=r_n(\alc,\bec)$ is a simple root of  $p_{\alc,\bec}$. Then $i\mu=r_m(\widetilde{\alpha},\widetilde{\beta})$, $m\neq n$, for some $\widetilde{\alpha}\geq\alc$ and $\widetilde{\beta}=f(\widetilde{\alpha})\in \overline{W(B)}$, where
\[
	f(\alpha):=-(c+d\mu+\mu^2)\left(1+\frac{\alpha}{\mu}\right).
\]
Moreover, $f(\alpha)\in \overline{W(B)}$ for all $\alpha\in [\alc,\widetilde{\alpha}]$, which implies that  $p_{\alpha,f(\alpha)}(i\mu)=0$ for all  $\alpha\in [\alc,\widetilde{\alpha}]$. Hence, there exists an $\epsilon'>0$ such that $r_m(\alc+\epsilon,f(\alc+\epsilon))=i\mu$, $m\neq n$, for $\epsilon\in (0,\epsilon')$ and continuity implies $r_m(\alc,\bec)=i\mu$, which contradicts  the assumption that $i\mu$ is a simple root.

Suppose $c+d\mu+\mu^2<0$ and that $i\mu$ is a double root of $p_{\alc,\bec}$ for $\bec<\beta_1$. Then Lemma \ref{3corhighlow} implies that $r_n(\widetilde{\alpha},\beta_1)=i\mu$ for some $\widetilde{\alpha}>\alc$. Due to continuity of $r_n(\alpha,\beta_1)$ in $\alpha$ it follows that there is some interval $J:=(i(\mu-\epsilon),i(\mu+\epsilon))$ such that $J\subset W_{[\alc+\gamma,\infty)\times\overline{W(B)}}(T)$ for $\widetilde{\alpha}-\alc>\gamma>0$ small enough. Since $i\mu$ is a critical point it follows that there is $m\neq n$ such that $r_m(\alpha,\bec)\notin J$ for $m\neq n$ and $\alpha>\alc$, but  $r_m(\alc,\bec)= i\mu$. Symmetry with respect to the imaginary axis and continuity then yields that $i\mu$ is a root of $p_{\alc,\bec}$ of order at least $3$. Hence we have a contradiction and conclude that $\bec=\beta_1$. A similar reasoning show that $\bec=\beta_0$ if $i\mu$ is a double root and $c+d\mu+\mu^2> 0$.
 
 Assume that $i\mu$ is a critical point  and that $i\mu$ is a root of $p_{\alc,\bec}$ of order at least $3$, then Lemma \ref{3en3root} implies that only two such cases can occur. Further, assume that  $c+d\mu+\mu^2<0$  (if $c+d\mu+\mu^2>0$ a similar proof can be used) and $\bec=\beta_+<\beta_1$, $\alc=\alpha_+$, $\mu=\mu_+\neq\mu_-$. 
 
From   \cite[Proposition 2.6]{ATEN} it follows that if $\omega\in \C$  satisfies the conditions
    \[
   	\alp(\omega):=\frac{(2\omega_\Im+d)|\omega|^4}{d|\omega|^2+2c\omega_{\Im}}\in \overline{W(A)},
    \]
    and
    \[
    \bet(\omega):=\frac{-2\omega_{\Im}\left((-\omega_{\Re}^2+\omega_{\Im}^2+d\omega_{\Im}+c)^2+\omega_{\Re}^2(2\omega_{\Im}+d)^2\right)}{d|\omega|^2+2c\omega_{\Im}}\in \overline{W(B)},
    \]
then $\omega\in W_\Omega(T)$. A straight forward computation shows that $i\mu_+$ is a minimum of $\alp$ and $\bet$ along the imaginary axis. This implies that if $\epsilon>0$ is chosen small enough then $\alp(\mu_++\epsilon)>\alpha_+$ and  $\bet(\mu_++\epsilon)\in (\beta_0,\beta_1)$. From continuity it then follows that there are some constants $\epsilon_2>0$ and  $\gamma>0$ such that the open ball $\mathcal{B}(i(\mu_++\epsilon),\epsilon_2)\subset W_{[\alc+\gamma,\infty)\times\overline{W(B)}}(T)$. 

Since $\alp(i\mu_+)=\alpha_+$, $\bet(i\mu_+)=\beta_+$, and $\epsilon>0$ can be chosen arbitrarily small, it follows that for some $n\in \{1,2,3,4\}$ and $\gamma>0$ there exists a continuous function $g: [\alpha_+,\alpha_++\gamma]\rightarrow [\beta_0,\beta_1]$ such that $g(\alpha_+)=\beta_+$, and $r_n(\alpha,g(\alpha))\in \C\setminus i\R$ for $\alpha\in(\alpha_+,\alpha_++\gamma]$. 
Symmetry with respect to the imaginary axis then implies that $r_n(\alpha_+,\beta_+)=r_m(\alpha_+,\beta_+)=i\mu_+$ for some $n\neq m$, where $r_m$ and $r_n$ for $\gamma>0$ small enough are in the same component of $W_{[\alpha_++\gamma,\infty)\times\overline{W(B)}}(T)$.  Further, $\beta<\beta_1$ implies that for  $\epsilon,\gamma>0$ small enough the interval $J:=(i(\mu_+-\epsilon),i(\mu_++\epsilon))$ is also in the same component of $W_{[\alpha_++\gamma,\infty)\times\overline{W(B)}}(T)$.
 Since  $i\mu_+$ is a critical point and the roots depend continuously on the parameters one of the roots $r_l(\alpha,\beta)$, $l\in \{1,2,3,4\}\setminus\{m,n\}$ has the properties $r_l(\alpha,\beta)\notin J$ for all $\alpha>\alpha_+$, $\beta\in \overline{W(B)}$, and $r_l(\alpha_+,\beta_+)=i\mu_+$. Hence, continuity implies that $r_l(\alpha,\bec)\in \C\setminus i\R$  for $\alpha=\alpha_++\epsilon$ with $\epsilon>0$ small enough. Symmetry with respect to the imaginary axis then implies that $i\mu_+$ is a quadruple root of $p_{\alpha_+,\beta_+}$. But then is $i\mu_+=i\mu_-$, which contradicts our initial assumption. 
The proof for the case when $c+d\mu+\mu^2>0$ is analogous except that  $i\mu$ is a local maximum of $\bet$, and the result has to be proven for $\beta>\beta_0$ instead of $\beta<\beta_1$.
\end{proof}

Lemma \ref{3lemdoui} narrows down the possible critical points to certain multiple roots of the polynomial $p_{\alc,\beta}$.
\begin{lem}\label{lemdoubleroot}
Assume that  $i\mu$ for $\mu\in \R$  is a double root of the polynomial $p_{\alpha,\beta}$ in \eqref{3pabdef}. Then $\alpha=0$ or $\alpha$ is  a root of the polynomial
\begin{equation}\label{3en2root}
q(\alpha)=\sum_{i=0}^4 q_i\alpha^i,
\end{equation}
where
\[	
\begin{array}{l}\label{eq:poly}
	q_0=4c(\beta  + c)^3( 4(\beta  + c)-d^2 ),\\
q_1=64(\beta  - c)c(\beta  + c)^2 - 4(\beta ^3 + 26\beta ^2c + 13\beta c^2 - 12c^3)d^2 + (\beta ^2 + 20\beta c - 8c^2)d^4,\\
q_2=4(24c^3 - 22c^2d^2 + 8cd^4 - d^6 + 3\beta ^2(8c - d^2) + \beta (-16c^2 - 13cd^2 + 5d^4)),\\
q_3=4(\beta (16c - 3d^2) - 2(8c^2 - 6cd^2 + d^4)),\\
q_4=4(4c - d^2).
\end{array}
\]
Unless $4c>d^2/4=\beta$ the converse statement also holds. Moreover, if $d^2/4=\beta$ then  $p_{c,\beta}$ has the double roots 
\[
	\omega=\pm \sqrt{c-\frac{d^2}{16}}-i\frac{d}{4}.
\]
\end{lem}
\begin{proof} 
The ansatz that $i\mu$ is a double root of $p_{\alpha,\beta}$ gives a system of equations in $\alpha$ that can be written in the form \eqref{3en2root}.
\end{proof}

\begin{lem}\label{3corhighc}\label{3cormidc}\label{3corlowc}\label{3corc}
Let $q$ and $\beta_\pm$ be defined as in \eqref{3en2root} and \eqref{3en3root}, respectively. 
\begin{itemize}
\item[(\rm i)]  
If $c>d^2/4$, then $q$ has two real roots that in the case $\beta\geq \beta_-$ correspond to purely imaginary double roots of $p_{\alpha,\beta}$, and $q$ has no real roots that correspond to purely imaginary double roots of $p_{\alpha,\beta}$ if $0\leq\beta< \beta_-$. 
\item[(\rm ii)] 
If $d^2/16<c< d^2/4$, then $q$ has two real roots  that correspond to purely imaginary double roots  of $p_{\alpha,\beta}$. 
\item[(\rm iii)] 
Assume $c\leq d^2/16$, if $\beta_+\leq\beta\leq \beta_-$ then $q$ has four real roots, and if $\beta<\beta_+$ or $\beta> \beta_-$ then $q$ has two roots.
\end{itemize}
\end{lem}
\begin{proof}
The  discriminant $\Delta_q$ of \eqref{3en2root} is
\begin{equation}\label{3disclem}
	\Delta_q=-2^4 3^9d^{16} (-4 \beta + d^2)^2 (\beta - \beta_+)^3 (\beta - \beta_-)^3,
\end{equation}
where $\beta_\pm$ is defined as \eqref{3en3root}. If $\beta=d^2/4$ then $\alpha=c$ is a double root but Lemma \ref{lemdoubleroot} shows that this root only corresponds to purely imaginary double roots of $p_{\alpha,\beta}$ if $c\leq d^2/16$.

{(\rm i)}
If $\beta>\beta_-$ then $\Delta_q<0$ and thus there are two real roots. For $\beta=0$ there is no real solution and since $\Delta_q>0$ for $\beta\in[0,\beta_-)$ the remaining statements follow by continuity.

{(\rm ii)}
For $\beta\neq d^2/4$ then $\Delta_q<0$ and for $\beta=d^2/4$ the roots can be computed explicitly.

{(\rm iii)}
For $c=0$ it holds that 
\[
	\beta_+=0<\frac{d^2}{4}<\frac{8}{27}d^2=\beta_-,
\]
and since $\beta_\pm\neq d^2/4$  for $c<d^2/16$ it follows by continuity that $\beta_+\leq d^2/4\leq \beta_-$, with equality if and only if $c=d^2/16$. Hence  $\Delta_q<0$ for $\beta\notin [\beta_+,\beta_-]$, and we conclude that there is two real roots. If $\beta=d^2/4$ the four real roots are
\[
	\alpha=c,\quad \alpha=\frac{32 c^2 - 56 c d^2 + 11 d^4 \pm d(5 d^2-16 c )^{\frac{3}{2}}}{8 (4 c - d^2)},
\]
where $\alpha=c$ is a double root. By continuity and that $\Delta_q>0$ for $\beta\in (\beta_+,\beta_-)\setminus \{d^2/4\}$ we conclude that there are four real roots.
\end{proof}

We now present sufficient conditions for accumulation of eigenvalues to the poles of $T$. The results are for the case  $d<2\sqrt{c}$ presented in Theorem \ref{3lemallB}, and for the case $d>2\sqrt{c}$ presented in Theorem \ref{3lemallBi}. 
\begin{thm}\label{3lemallB}
Let $T$ denote the operator function \eqref{tdef}. Assume that $d<2\sqrt{c}$ and let $\alpha_-$, $\beta_-$ be the constants \eqref{3en3root}. Moreover,  $\alpha_1$ is the largest real root of the polynomial \eqref{3en2root}, with $\beta=\beta_0$ and $\alpha_\infty$ denotes any real number. Further assume that $\overline{W(A)}\subset(\alc,\infty)$, where $\alc$ is given by: 
\begin{equation}\label{3Atable}
	\begin{array}{c c l | c  }
	&&& \alc\\
	\hline\rule{0pt}{4ex}  
	 \dfrac{d^2}{4}\in \overline{W(B)}&&&c\rule[-2.2ex]{0pt}{0pt}\\
	 \hline\rule{0pt}{4ex}  
	\beta_0> \dfrac{d^2}{4}&&&0\rule[-2.2ex]{0pt}{0pt}\\
	 \hline\rule{0pt}{4ex}  
	  \beta_1<\dfrac{d^2}{4}&\&& \beta_-\in \overline{W(B)}& \alpha_-\rule[-2.2ex]{0pt}{0pt}\\
	 \hline\rule{0pt}{4ex}  
	 \beta_1<\dfrac{d^2}{4}&\&& \beta_0>\beta_- & \alpha_1\rule[-2.2ex]{0pt}{0pt} \\
	   \hline\rule{0pt}{4ex}  
	 \ \beta_1<\dfrac{d^2}{4}&\&&\beta_1< \beta_-&\alpha_\infty\\
	\end{array}\ .
\end{equation}
Then there is a branch of eigenvalues accumulating at $\delta_\pm$ if and only if $\Hs\bot\ker B$ is infinite dimensional. Set $\Hs_0:=\ker (A-\alc)^{-\frac{1}{2}}B(A-\alc)^{-\frac{1}{2}}$ and let $\Hs_1:=\Hs\bot\Hs_0$ denote the Hilbert spaces with the inner products $\ip{\cdot}{\cdot}_{\Hs_i}=\ip{V_i\cdot}{V_i\cdot}$, where $V_i$ is defined in \eqref{1partis}. Let $\{u_j\}$ denote the set of eigenvectors and associated vectors corresponding to the branch of accumulating eigenvalues of $T$. Then the set $\{V_1^*(A-\alc)^\frac{1}{2}u_j\}$ is complete and minimal in $\Hs_1$.
\end{thm}
\begin{proof}
Define $\Omega:=\overline{W(A)}\times \overline{W(B)}$ and let $p_{\alpha,\beta}$ and $W_\Omega(T)$ be defined as in \eqref{3pabdef} and in \eqref{3endef}, respectively.
We want to show that there is a simply connected bounded open set $\Gh_\pm\subset\C$ such that $\delta_\pm\in \Gh_\pm$, $\delta_\mp\notin \Gh_\pm$, and  $\partial\Gh_\pm\cap W_\Omega(T)=\emptyset$.

If the lower bound on $A$ is large enough, $W_\Omega(T)$ consists of four components, where $\delta_\pm$ is in one of the bounded components. Hence, $\Gh_\pm$ with the desired properties exists for some lower bound on $W(A)$. Now we want to show that this property holds if $\alc$ is given by \eqref{3Atable}. 

Assume $d^2/4\in \overline{W(B)}$,  then it follows from Lemma \ref{3lemdou} that there is a critical point that is not on the imaginary axis if and only  if $\alc=c$. Hence, if $\ac>c$ it follows that $\ac>0$ and thus there are no solutions on the imaginary axis, which together with Lemma \ref{3lemdou} implies that $W_\Omega(T)$ consists of four disjoint components. 

Assume that  $\beta_0>d^2/4$, then Lemma \ref{3lemdou} implies that there is no critical point that is not on the imaginary axis. Hence, if $\ac>0$ there is no purely imaginary solutions and thus  $W_\Omega(T)$ consists of four disjoint components. 

Assume that $d^2/4>\beta_1$, then it follows from \cite[Proposition 3.17]{ATEN} that  the components containing $\delta_+$ and $\delta_-$ are disjoint from the rest of $W_{[0,\infty)\times\overline{W(B)}}(T)$. Hence, $\omega=0$ is not a critical point for the components containing the poles and thus $\alc$ can be negative. 

If $\beta_-\in \overline{W(B)}$, then Lemma \ref{3lemdoui} implies that there can be a critical point $\omega$ at $\alpha_-$ and in that case $\omega$ is a double root of $p_{\alc,\beta_0}$. From Lemma \ref{3corc} {(\rm i)} it follows that, $p_{\alc,\beta_0}$ has no double roots if $\beta_0<\beta_-$. If $\beta_0=\beta_-$ then Lemma \ref{lemdoubleroot} implies that if $p_{\alc,\beta_0}$ has a double root then  $\alc=\alpha_-$. Hence, if $\ac>\alpha_-$ then the components containing $\delta_\pm$ are disjoint from the rest of $W_\Omega(T)$.

If $\beta_0>\beta_-$, then Lemma \ref{3lemdoui} implies that any critical point must be a double root of $p_{\alc,\beta_0}$. From Lemma \ref{3corc} {(\rm i)} it then follows that there are two such roots, $\alpha_1$ and $\alpha_2$, where $\alpha_1>\alpha_2$. Hence, 
if  $\ac>\alpha_1$ then the components containing $\delta_\pm$ are disjoint from the rest of $W_\Omega(T)$.

If $\beta_1<\beta_-$, Lemma \ref{3lemdoui} implies that any critical point must be a nonzero double root of $p_{\alc,\beta_0}$, but it follows from Lemma \ref{3corc} {(\rm i)} that there are no such roots. Hence, in this case any lower bound on $A$ is sufficient. 

We have now shown that for the lower bounds on $A$ given by \eqref{3Atable} the conditions of Lemma \ref{mobroot} is satisfied. Hence, it follows that $P_\pm$ has a spectral root around $0$. From definition of \eqref{pen} it can be seen that it has the form of \eqref{Adef2} with $k=1$ and $K=0$. From Lemma \ref{1lemext} it then follows that the linear operator function $C$, defined in the lemma, has the same spectrum as $P_\pm$ in $\C\setminus\{0\}$, where
\[
	C(\lambda):=\lambda+(I_{\Hs_1}+K_1)H_1,
\]
 $H_1=V_0^*(A-\alc)^{-\frac{1}{2}}B(A-\alc)^{-\frac{1}{2}}V_0$, and $K_1$ is compact. From $K=0$, Lemma \ref{2schurlem}, and Lemma \ref{3invlem} it follows that $I_{\Hs_1}+K_1$ is invertible. The result then follows from Theorem \ref{main}. 
\end{proof}
Figure \ref{3fig:disjp} depicts $W_\Omega(T)$ for cases where Theorem \ref{3lemallB} can be used to prove accumulation of complex eigenvalues. The dots denote $\delta_\pm$ and the dashed curve in each of the panels is the boundary of a bounded $\Gh$ satisfying $\delta_+\in \Gh$, $\delta_-\notin \Gh$, and  $\partial\Gh\cap W_\Omega(T)=\emptyset$.
\begin{center}
\begin{figure}
\includegraphics[width=12.5cm]{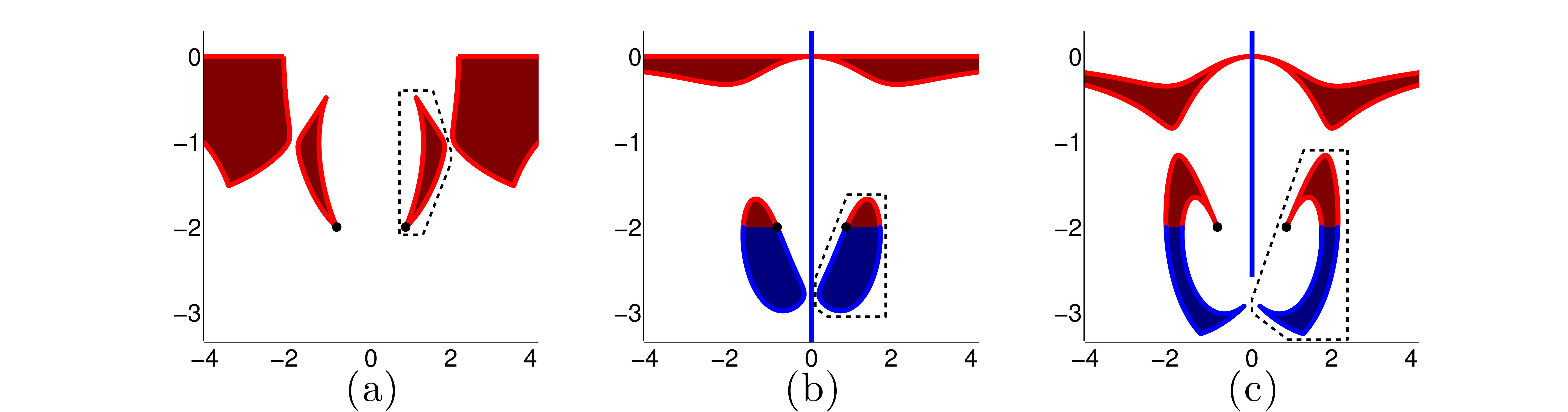}
\caption{Examples of $W_\Omega(T)$ in red and blue, with $c=19/4$ and $d=4$, where the dashed line separates a component from the rest of $W_\Omega(T)$. In (a) $d^2/4\in \overline{W(B)}$.  In (b) $\beta_1<\min(d^2/4,\beta_-)$.  In (c) $\beta_1<d^2/4$ and $\beta_0>\beta_-$.}\label{3fig:disjp}
\end{figure}
\end{center}
\begin{thm}\label{3lemallBi}
Let $T$ denote the operator function \eqref{tdef} and assume that $d>2\sqrt{c}$. Order the real roots  $\alpha_1,\alpha_2,\dots$ of \eqref{3en2root} with $\beta=\beta_1$ decreasingly and let $\alpha_\pm$ and $\beta_\pm$ denote the constants in \eqref{3en3root}. Further assume that $\overline{W(A)}\subset(\alc,\infty)$, where $\alc$ is given by:
\begin{equation}\label{3Aitable}
	\begin{array}{ l c c | c |c   }
	 &&&\alc \text{ for } \delta_- &\alc \text{ for}\ \delta_+\\
	\hline\rule{0pt}{4ex}  
	 \dfrac{d^2}{16}<c&\text{or}&\beta_1<\beta_+& \alpha_1 & \alpha_1^{\dagger}\rule[-2.2ex]{0pt}{0pt}\\ 
	 \hline\rule{0pt}{4ex}
	\dfrac{d^2}{16}\geq c& \& & \beta_+\leq\beta_1<\dfrac{d^2}{4}& \alpha_3 & \alpha_2\rule[-2.2ex]{0pt}{0pt}\\
	 \hline\rule{0pt}{4ex}
	\dfrac{d^2}{16}\geq c&\&  & \dfrac{d^2}{4}\leq\beta_1\leq\beta_-&\alpha_2& \alpha_3\rule[-2.2ex]{0pt}{0pt} \\
	 \hline\rule{0pt}{4ex}
	\dfrac{d^2}{16}\geq c& \&  &  \beta_0\leq\beta_-<\beta_1&\max(\alpha_1,\alpha_-)&\alpha_1\rule[-2.2ex]{0pt}{0pt}\\
	 \hline\rule{0pt}{4ex}
	\dfrac{d^2}{16}\geq c&\&  &\beta_-<\beta_0& \alpha_1^{\ddagger}& \alpha_1
	\end{array}\ .
\end{equation}
\small{$\dagger$ If $0\in \overline{W(B)}$ and $\beta_1<\beta_+$  then $\alc=\max(\alpha_1,\delta_+^2)$. 
\\
$\ddagger$ 
If $\beta_-<B$ and $c=0$ then $\alc$ can be any real number less than $\ac$.\\ \\ }
\normalsize
Then there is a branch of eigenvalues accumulating at $\delta_\pm$ if and only if $\Hs\bot\ker B$ is infinite dimensional. If $c=0$ then $\delta_+$ is a removable singularity of $T$ and thus there is no accumulation to $\delta_+$ in that case. Let $\Hs_0:=\ker (A-\alc)^{-\frac{1}{2}}B(A-\alc)^{-\frac{1}{2}}$ and $\Hs_1:=\Hs\bot\Hs_0$ denote the Hilbert spaces with the inner products $\ip{\cdot}{\cdot}_{\Hs_i}=\ip{V_i\cdot}{V_i\cdot}$ where $V_i$ is defined in \eqref{1partis}. Let $\{u_j\}$ denote the set of eigenvectors and associated vectors corresponding  to the branch of accumulating eigenvalues  of $T$, then the set 
$\{V_1^*(A-\alc)^\frac{1}{2}u_j\}$ is complete and minimal in $\Hs_1$.
\end{thm}

\begin{proof}
This proof follows the same pattern as the proof of Theorem \ref{3lemallB}: First we show that for $\alc$ given in \eqref{3Aitable}, the components of $W_\Omega(T)$ containing  $\delta_+$ and $\delta_-$, respectively, is disjoint from the rest of $W_\Omega(T)$.

Since for $A$  large enough there are four components and the components with the poles are two disjoint subsets of $i\R$, the critical points can be divided into two groups: either the components containing the poles merge, or the unbounded components   merge with the component containing $\delta_+$ or $\delta_-$. For the proof note that $\alpha_+\leq\alpha_-$ and $\beta_+\leq\beta_-$. 

Assume that $B>0$. For $i\mu\in i\R\setminus i[ \Im \delta_-,\Im \delta_+]$, close to $\delta_+$ or $\delta_-$, it holds that $p_{\alpha,\beta}(i\mu)>0$ for $(\alpha,\beta)\in\Omega$. Hence, critical points on $i\R$ must belong to the purely imaginary interval $i(\Im \delta_-,\Im\delta_+)$. Lemma \ref{3lemdoui}, then states that $i \mu$ is either a root of order at least $3$ of $p_{\alc,\beta_-}$ or a root of order at least $2$ of $p_{\alc,\beta_1}$. Hence, $\alpha_1,\alpha_2,\hdots$ and $\alpha_-$ are the only possible $\alc$, where components containing the poles can merge. 

 If $d^2/16<c$ or $B<\beta_+$, then Lemma \ref{lemtripleroot} implies that there is no triple root. Hence, if $\alc=\alpha_1$ then the components containing $\delta_\pm$ are disjoint from the rest of $W_\Omega(T)$.
 
 If $\beta_+\leq\beta_1\leq\beta_-$, then it follows from Corollary \ref{3corlowc} that  \eqref{3en2root} has  four real roots for $\beta=\beta_1$. 
 These roots are the only possible $\alc$ and we show below which of $\alpha_1,\alpha_2,\alpha_3,\alpha_4$ that corresponds to the merge. If  $d^2/16=c$, then $\beta_+=\beta_-=d^2/4$, $\alpha_1=\alpha_2=\alpha_3=c$ and there is nothing to prove. Hence, assume $d^2/16>c$.
If $\beta_1=d^2/4$ then it follows that $\alpha_2=\alpha_3=c$ and
 \[
 	\alpha_1=c+\frac{ 48 c d^2 - 11 d^4 + d(5 d^2-16 c )^{\frac{3}{2}}}{8 ( d^2-4c)},\quad \alpha_4=c+\frac{48 c d^2 - 11 d^4 - d(5 d^2-16 c )^{\frac{3}{2}}}{8 (d^2-4c)}.
 \]
 These are the values of $\alpha$ where $p_{\alpha,\beta_1}$ has a double root. By straight forward computations it can be seen that
\[
	\Im \delta_+>\mu_2> \mu_1> \mu_3>\Im \delta_->\mu_4,
\]
where $\mu_i$ is the double root of $p_{\alpha_i,\beta_1}$. Lemma \ref{3corhighlow} then  yields that 
\[
	i\mu_2,i\mu_3\notin W_{[\alpha_1,\infty)\times \overline{W(B)}}(T),
\]
and thus the components containing the poles does not merge if $\alc=\alpha_1$.  Lemma~\ref{3disclem} then yields that the merges with the components containing $\delta_\pm$ are $\alc=\alpha_2$ and $\alc=\alpha_3$. If $\beta_1=\beta_+$ then $\alpha_+$ is a double root of $q$ and the corresponding triple root of $p_{\alpha_+,\beta_1}$ is $i\mu_+$. By computing the remaining two roots of $q$ it follows that $\alpha_+=\alpha_1=\alpha_2$ are the largest solutions. Further it can be seen that the corresponding double roots $i\mu_3$ of $p_{\alpha_3,\beta_1}$ and $i\mu_4$ of $p_{\alpha_4,\beta_1}$ satisfies $\mu_3,\mu_4<\mu_+$. Hence, for $\beta_1=\beta_+$ the merge with $\delta_+$ is for $\alc=\alpha_2$ and  the merge with $\delta_-$ is for $\alc=\alpha_3$. Similar reasoning for $\beta_1=\beta_-$ shows that the merge with $\delta_-$ is for $\alc=\alpha_2$ and  the merge with  $\delta_+$ is for $\alc=\alpha_3$. Hence, for some $\beta_1\in(\beta_+,\beta_-)$, the values $\alc=\alpha_2=\alpha_3$ give a merge for the component containing $\delta_+$ and for the component containing $\delta_-$. It then follows that $q$ has a double root and thus from \eqref{3disclem} we obtain $\beta_1=d^2/4$.
 
If $ \beta_-\in W(B)$, $\alpha_1$ and $\alpha_2$ are the only two real solutions of \eqref{3en2root}. There are two different cases: either $\alpha_1\geq\alpha_-$ or $\alpha_1<\alpha_-$. If $\alpha_1\geq\alpha_-$, then there is nothing to show. However, if $\alpha_1<\alpha_-$ there will be a merge with $\alc=\alpha_-$ at some point on $i\R$. Due to continuity and the result for $\beta_+\leq\beta_1\leq\beta_-$ the merge is with the component containing $\delta_-$. Hence, in this case, for $\alc=\alpha_-$ and for $\alc=\alpha_1$, the components containing $\delta_-$ respectively $\delta_+$ are disjoint from the rest of $W_\Omega(T)$.

If $\beta_-\leq B$, then there is no triple root, and the result follows as in the case $B<\beta_+$. Furthermore, if $c=0$ then zero is a removable singularity of $T$ and thus
\[
	T(\omega)=A-\omega^2-\frac{\omega}{-id-\omega}B, \quad  \omega\in \C\setminus\{\delta_+,\delta_-\}.
\]
This reduced function has only one pole and thus no merge of  components containing poles. Hence, Lemma \ref{3corc} {(\rm iii)} implies that $\alc=\alpha_1=0$ cannot give a critical point. Furthermore, $-d=\Im\delta_-$ is the minimal imaginary part of the component containing $\delta_-$ and the double root of $p_{\alpha_2,\beta}$ is  $i\mu_2,$ where $\mu<-d$. Thus for $\alc=\alpha_2$ there is no merge with the component containing $\delta_-$. Since there is no other possible $\alc$ it follows that $W(A)$ bounded from below is a sufficient condition on $A$.

Assume that $0\in W(B)$ and $\alc<0$. Then the only property that has to be modified in the proof is that  $i[-\Im\sqrt{\ac},\Im \sqrt{\ac}]\subset W_\Omega(T)$. Hence, $\delta_+$  is connected to the unbounded component if $\delta_+^2\in W(A)$ and we must thus show that if $ d^2/16\geq c$ and $\beta_1\geq \beta_+$ the tabular \eqref{3Aitable} implies that  $\ac>\delta_+^2$. 
 
Assume that $\beta_1>\beta_-$, then since $\beta_->d^2/4$ it follows that $q_0>0$, $q_4<0$. Moreover, $q$ has two real roots and $\alpha_1>0$. Now we consider the polynomial $q$ as a continuous function in $\beta_1$ and observe what happens when we decrease $\beta_1$.  The number of non-positive roots can only increase when $q_0=0$ and when the number of real roots of $q$ increases. The former happens when $\beta_1=d^2/4-c$ while due to Lemma \ref{3corc}, the latter happens when $\beta_1=\beta_-$. However, for $\beta_1=\beta_-$, the double root of $q$ is $\alpha_->0$. Hence, there is at most one negative root of $q$ for $\beta_1>d^2/4-c$. Furthermore, since the root $0$ of $q$ is simple when $\beta_1=d^2/4-c$, there is at most two negative roots of $q$ for $\beta_1\leq d^2/4-c$. This implies that the lower bound on $\overline{W(A)}$ for $\delta_+$ given in \eqref{3Aitable} is positive if  $d^2/16\geq c$ and $\beta_1\geq \beta_+$ and thus $\alc>\delta_+^2$ in these cases.

The rest of the proof is now completely analogous to proof of Theorem \ref{3lemallB}, with the exception that Lemma \ref{3invlemi} is used instead of Lemma \ref{3invlem}. The additional condition that $A>\delta_\pm^2$ in the case when $\ker B\neq \{0\}$ is enforced by \eqref{3Aitable} for $\delta_+$ and a consequence of Lemma \ref{3corhighlow} for $\delta_-$.
\end{proof}
\begin{center}
\begin{figure}
\includegraphics[width=12.5cm]{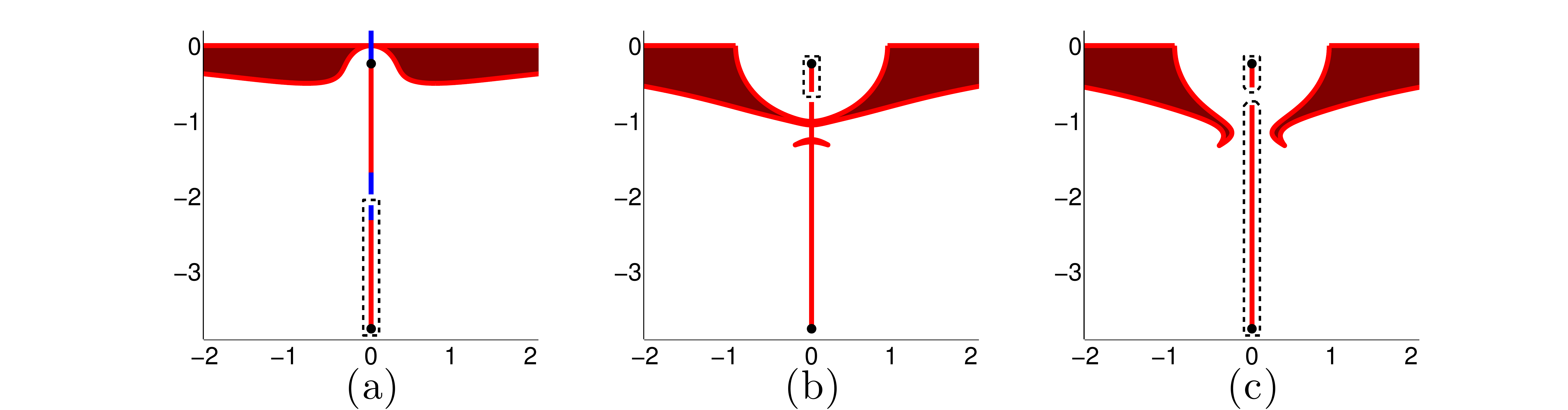}
\caption{Examples of $W_\Omega(T)$ in red and blue, with $c=9/10$ and $d=4$, where the dashed line separates a component from the rest of $W_\Omega(T)$. In (a) $\beta_1<\beta_+ $ and $0\in \overline{W(B)}$. In (b)   $d^2/4\leq\beta_1\leq\beta_-$. In (c) $\beta_-\in\overline{W(B)}$. }\label{3fig:disjm}
\end{figure}
\end{center}

The condition  $\ac>\delta_+^2$ is not implied by $\ac>\alpha_1$ in the case when $\beta_1<\beta_+$. This is illustrated in Figure \ref{3fig:disjm}.(a), where $\alpha_1<\alpha_0<\delta_+^2$ and only the component containing $\delta_-$ is disjoint. In Figure  \ref{3fig:disjm}.(b), the conditions   $\alpha_3<\alpha_0<\alpha_2$ are fulfilled,  which implies that the component containing  $\delta_+$ is disjoint. In Figure  \ref{3fig:disjm}.(c),  the conditions $\alpha_-<\alpha_1<\alpha_0$ are fulfilled, which implies that both of the components containing a pole are disjoint from the rest of $W_\Omega(T)$. In each panel, the dots denote $\delta_\pm$ and each dashed curve  is the boundary of a bounded set $\Gh_\pm$ satisfying $\delta_\pm\in \Gh_\pm$, $\delta_\mp\notin \Gh_\pm$, and  $\partial\Gh_\pm\cap W_\Omega(T)=\emptyset$. In Figure  \ref{3fig:disjm}.(c) both $\Gh_+$ and $\Gh_-$ are visualized.

\subsection{The case $d=2\sqrt{c}$}
In this section we will present sufficient conditions on $T$  that assure that Theorem \ref{main2} is applicable, and thus prove the existence of an infinite sequence of eigenvalue accumulating to $\delta_\pm=-id/2$. This is done similarly as in the case when $d\neq2\sqrt{c}$.

\begin{lem}\label{ratpolt}
Assume that $d=2\sqrt{c}$ and let $\widehat{P}$ denote the operator polynomial \eqref{transb}. Define the polynomial $P$ for the pole $-id/2$ as in \eqref{penmult}. Then 
\begin{equation}\label{pent}
P(\lambda):=\sum_{j=0}^4 \lambda^jP_j,\quad \lambda\in \C,
\end{equation}

where $H:=-\frac{d^2}{4} (A-\alc)^{-\frac{1}{2}}B(A-\alc)^{-\frac{1}{2}}$ and
\begin{equation}\label{Ai02c}
	\begin{array}{l}
	P_0:=H,\\
	P_1:= \dfrac{4i}{d}H,\\
	P_2:=I+\left(\alc+ \dfrac{d^2}{4}\right)(A-\alc)^{-1}-\dfrac{4}{d^2}H,\\
	P_3:=id(A-\alc)^{-1},\\
	P_4:=-(A-\alc)^{-1}.
	\end{array}
\end{equation}
\end{lem}
\begin{proof}
Follows directly from \eqref{penmult}.
\end{proof}

\begin{lem}\label{3invlemc}

Assume $d=2\sqrt{c}$ and let $P$ be defined as  \eqref{pent}. Set $\Hs_0:=\ker H$ and define $V_0$ as in \eqref{1partis}. 
If $\ac>-d^2/4$ then $V_0^* P_2V_0$ is invertible.
\end{lem}
\begin{proof}
From \eqref{Ai02c} it follows that
\[
	V_0^*P_2V_0=I_{\Hs_0}+\left(\alc+\frac{d^2}{4}\right)V_0^*(A-\alc)^{-1}V_0,
\]
and the rest of the proof is similar to the proof of Lemma \ref{3invlemi}.
\end{proof}

\begin{lem}\label{ratinf}
Assume that $d=2\sqrt{c}$ and let $P$ denote the operator polynomial \eqref{pent}. Define the operator polynomial
\begin{equation}\label{pentinf}
	\widetilde{P}(\lambda):=H' -\lambda i d H' +\lambda^2 \left(I-\left(\alc+\frac{d^2}{4}\right)H' -\frac{4}{d^2}H  \right)  + \lambda^3 \frac{4i}{d}H  +\lambda^4 H,
\end{equation}
where $H':=-(A-\alc)^{-1}$. Then $ \widetilde{P}(\lambda)=\lambda^4P(\lambda^{-1})$,  for $\lambda\in\C\setminus\{0\}$ and
\[
	\sigma(\widetilde{P})=\{0\}\cup\{\lambda^{-1}:\lambda\in\sigma(P)\setminus\{0\}\},\quad
	\overline{W(\widetilde{P})}=\{0\}\cup\{\lambda^{-1}:\lambda\in\overline{W(P)}\setminus\{0\}\}.
\]
\end{lem}
\begin{proof}
The claims follow directly from \eqref{pent} and \eqref{ratpolt}.
\end{proof}

\begin{lem}\label{mobrooti}
Let $T$ be defined as in \eqref{tdef},  set $\Omega:=\overline{W(A)}\times \overline{W(B)}$ and let $W_\Omega(T)$ denote the enclosure \eqref{3endef}. Assume that $d=2\sqrt{c}$  and that there is an open  disc $\Gh\subset\C$ such that $-id/2\in \Gh$ and  $\partial\Gh_\pm\cap W_\Omega(T)=\emptyset$. 
Then there are spectral divisors of order $2$ of the operator polynomials \eqref{pent} and  \eqref{pentinf} on some open disks $\G$ and $\G'$, respectively, with $0\in \G$ and $0\in \G'$.
\end{lem}

\begin{proof}
For \eqref{pent} the claim is a consequence of Lemma \ref{mobrootmult} and the property $W_\Omega(T)\supset \overline{W(T)}$.
From Lemma \ref{ratinf}, a similar argument holds  for  \eqref{pentinf} with $\overline{\G'}:=\{0\}\cup\{\lambda^{-1}:\lambda\in\C\setminus\G\}$.
\end{proof}

 In Theorem \ref{3lemallBt}, we present sufficient conditions for accumulation of eigenvalues to the pole of $T$ in the case $d=2\sqrt{c}$. 
\begin{thm}\label{3lemallBt}
Let $T$ denote the operator function \eqref{tdef} and take $c=d^2/4$. 
Further assume that $\overline{W(A)}\subset(\alc,\infty)$, where $\alc$ is given by:
\begin{equation}\label{3Actable}
	\begin{array}{ c | c  }
	&\alc\\
	\hline \rule{0pt}{4ex} 
	 \dfrac{d^2}{4}\in \overline{W(B)}&\dfrac{d^2}{4}\rule[-2.2ex]{0pt}{0pt} \\
	 \hline\rule{0pt}{4.3ex} 
		\beta_0>\dfrac{d^2}{4}&-\dfrac{512 \beta_0^2 - 120 \beta_0 d^2 + 3 d^4 +4 (64 \beta_0 - 7 d^2) \sqrt{ 4 \beta_0^2 - \beta_0d^2}}{16d^2}\rule[-2.2ex]{0pt}{0pt}\\
	 \hline\rule{0pt}{6ex} 
	\beta_1< \dfrac{d^2}{4}&-\dfrac{\left(d+\sqrt{d^2-4\beta_1}\right)^2}{16}\left(\dfrac{\beta_0\left(d+\sqrt{d^2-4\beta_1}\right)^2}{\beta_1^2}+1\right)\\
	\end{array}\ .
\end{equation}
Then there is a branch of eigenvalues accumulating at $-id/2$ if and only if $\Hs\bot\ker B$ is infinite dimensional. If $\Hs\bot\ker B$ is finite dimensional there are $2 \dim \Hs\bot\ker B$ eigenvalues in the branch of eigenvalues. Furthermore, there is a branch of accumulating eigenvalues to complex $\infty$.
\end{thm}

\begin{proof}
The proof is similar to the proofs of Theorem \ref{3lemallB} and of Theorem \ref{3lemallBi}. However,  to be able to utilize Lemma \ref{mobrooti}, $\Gh$ has to be a disk. Assume that $\frac{d^2}{4}\in \overline{W(B)}$ and let $\Gh$ denote the disc 
\[
 	\Gh:=\left\{\omega\in\C:\left|\omega+i\frac{d}{2}\right|<\frac{d}{2}\right\}.
\]
Assume $\omega\in W_\Omega(T)\cap \partial \Gh$, where $W_\Omega(T)$ is defined  in \eqref{3endef}. From \eqref{3Actable} it follows that  $\ac>0$, which implies $\omega\notin i\R$. Then \cite[Proposition 2.6]{ATEN} yields that
\begin{equation}\label{prop26cond}
	\frac{d^2}{4}\in W(B),\quad -4\left(\Im\omega+\frac{d}{2}\right)\Im\omega\in W(A)
\end{equation}
for $\omega\in W_\Omega(T)\cap\partial \Gh$. However, since  $-4\left(\Im\omega+d/2\right)\Im\omega\leq d^2/4$, by assumption we have $\omega\in W_\Omega(T)\cap \partial \Gh=\emptyset$. 

Assume $\frac{d^2}{4}<B$ and define 
\[
	\Gh_r:=\left\{\omega\in\C:\left|\omega-i\left(r-\frac{d}{2}\right)\right|<r\right\}.
\]
If $\omega\in W_\Omega(T)\setminus i\R \cap \partial\Gh_{\frac{d}{2}}$, then \cite[Proposition 2.6]{ATEN} implies
\[
	\frac{d^2}{4}\in \overline{W(A)},\quad -4\left(\Im \omega+\frac{d}{2}\right)\Im \omega\in \overline{W(B)},
\]
and since $-4(\Im \omega+d/2)\Im \omega\leq d^2/4$ by assumption we conclude that $W_\Omega(T)\setminus i\R \cap \partial\Gh_{\frac{d}{2}}=\emptyset$.

The point $0\in \Gh_{\frac{d}{2}}$ is a double root of $p_{0,\beta}$ for all $\beta\in \overline{W(B)}$ and $\pm id/2$ is not a root of $p_{\alpha,\beta}$ for $\alpha\geq 0$. Hence if $r(\alpha,\beta)$ for some pair $(\alpha,\beta)$, $\alpha\geq0$ is a root of $p_{\alpha,\beta}$ in $\Gh_{\frac{d}{2}}$. Then $r(\alpha,\beta)\in\Gh_{\frac{d}{2}}$ for all $(\alpha,\beta)$ and the root will approach the pole $-id/2$ as $\alpha\rightarrow\infty$. Since exactly two roots of $p_{\alpha,\beta}$ approach $-id/2$ as $\alpha\rightarrow\infty$
there cannot exist points in $W_\Omega(T)\setminus \Gh_{\frac{d}{2}}$ arbitrarily close to $-id/2$.
Moreover, since $W_\Omega(T)$ is a closed set in $\C\setminus\{ -id/2\}$ this yields that there must be some smallest $r>d/2$ such that $W_\Omega(T)\setminus i\R \cap \partial\Gh_{r}\neq\emptyset$ holds.  Define $\Omega'=\R\times \overline{W(B)}$, then the boundary of $\overline{W_{\Omega'}(T)\setminus i\R}$ is  $\overline{W_{\R\times \beta_0}(T)\setminus i\R}\cup\overline{W_{\R\times \beta_1}(T)\setminus i\R}$. The boundary is smooth since it is given by the roots of $p_{\alpha,\beta_0}$ and of $p_{\alpha,\beta_1}$
and these polynomials do not have double roots in $\C\setminus i\R$ for $\alpha\in \R$. This means that for the minimum $r>d/2$ such that $W_\Omega(T)\setminus i\R \cap \partial\Gh_{r}\neq\emptyset$ it must hold that $\partial \Gh_r$ tangents $\overline{W_{\R\times \beta_0}(T)\setminus i\R}$ or  $\overline{W_{\R\times \beta_1}(T)\setminus i\R}$.  The ansatz that for some $r_0$, the disc $\Gh_{r_0}$ and  $W_{\R\times \beta_0}(T)$ have the same tangent  in $W_{\R\times \beta_0}(T)\setminus i\R\cap \partial\Gh_{r_0}$ then implies
\begin{equation}\label{eq:rvalue}
	r_0=\frac{2\beta_0+\sqrt{4\beta_0^2-\beta_0d^2}}{d},
\end{equation}
which is the smallest value such that $W_{\Omega'}(T)\setminus i\R \cap \partial\Gh_{r_0}\neq\emptyset$. Since $W_\Omega(T)\setminus i\R\subset W_{\Omega'}(T)\setminus i\R$ it follows trivially that for $r<r_0$ we have that $W_{\Omega}(T)\setminus i\R \cap \partial\Gh_{r}=\emptyset$. In the nest step we consider the points on the imaginary axis:
\[
	\Gh_{r_0}\cap i\R=\left\{-i\frac{d}{2},i y_0 \right\}, \quad y_0:=\frac{4\beta_0+2\sqrt{4\beta_0^2-\beta_0d^2}}{d}-\frac{d}{2}.
\]
From the condition $\ac>\alc$ given by \eqref{3Actable} it follows that $i y_0\notin W_{\Omega}(T)$. Since there are no points in $W_\Omega(T)\setminus \Gh_{\frac{d}{2}}$ arbitrarily close to $-id/2$ it follows that there are no points in  $W_\Omega(T)\setminus \Gh_{r_0}$ arbitrarily close to $-id/2$. Hence,  from the closeness of $W_\Omega(T)$ in $\C\setminus\{-id/2\}$, there exist some $\epsilon_1>0$ such that $W_{\Omega}(T)\setminus i\R \cap \partial\Gh_{r_0-\epsilon}=-id/2$ for all $2\epsilon_1\geq\epsilon>0$. For $0<\epsilon_2\leq\epsilon_1$ it follows that $W_{\Omega}(T)\setminus i\R \cap \partial\Gh_{r_0-\epsilon_1-\epsilon_2}=-id/2$. Thus if we choose $\epsilon_2>0$ small enough and define the disc as
\[
	\Gh:=\left\{\omega\in\C:\left|\omega-i\left(r_0-\frac{d}{2}-\epsilon_1-2\epsilon_2\right)\right|<r_0-\epsilon_1\right\},
\]
then $\Gh\cap W_{\Omega}(T)=\emptyset$.

Assume that $B<d^2/4$ and take $\omega\in \C\setminus i\R$ that satisfies
\begin{equation}\label{3tstrip}
	\mu:=\frac{-d-\sqrt{d^2-4\beta_1}}{4}<\omega_\Im<\frac{-d+\sqrt{d^2-4\beta_1}}{4}. 
\end{equation}
Then straight forward computations show that $\omega\notin W_\Omega(T)$. From the definition of $p_{\alpha,\beta}$ the point $i\mu$ is in $W_\Omega(T)$ if and only if
\[
	-\frac{\left(d+\sqrt{d^2-4\beta_1}\right)^2}{16}\left(\frac{\beta\left(d+\sqrt{d^2-4\beta_1}\right)^2}{\beta_1^2}+1\right)\in W(A),
\]
for some $\beta\in W(B)$. Hence, if $W(A)$ satisfies the lower bound in \eqref{3Actable}, then $i\mu\notin W_\Omega(T)$. Further, since $W_\Omega(T)$ is a closed set in $\C\setminus \{-id/2\}$ and there are no  $\omega\in W_\Omega(T) \setminus i\R$ that satisfy \eqref{3tstrip}, there exists an $\epsilon>0$ such that $\omega\notin W_\Omega(T)$ for all $\omega$ that satisfies $\mu<\omega_\Im<\mu+\epsilon$. Furthermore, the component of $W_\Omega(T)$ with imaginary part below $\mu+\epsilon/2$ is bounded. Hence, $-id/2\in \Gh$  and $W_\Omega(T)\cap\partial\Gh=\emptyset$, for 
\[	
 	\Gh:=\left\{\omega\in\C:\left|\omega+i\left(\mu+\frac{\epsilon}{2}+r\right)\right|<r\right\}
\]
with $r>0$ large enough.

Hence, if $\alc$ is given by the tabular \eqref{3Actable} we can always find an open disc $\Gh$ such that $-id/2\in\Gh$ and $W_\Omega(T)\cap\partial\Gh=\emptyset$. Lemma \ref{mobrooti} then implies that $P$ has a spectral divisor of order $2$ on  $\G$ with $0\in\G$. From \eqref{Ai02c} it can be seen that $P$ has the form of \eqref{Adef2} with $k=2$, $K=0$, and $\ker P_1=\ker H$. Lemma \ref{1lemext} then yields that the operator polynomial $C$ has the same spectrum as $P$ in  $\C\setminus\{0\}$, and the structure
\[
	C(\lambda):=\lambda^2+C_1\lambda+(I_{\Hs_1}+K_1)H_1,
\]
where $H_1=V_0^*(A-\alc)^{-\frac{1}{2}}B(A-\alc)^{-\frac{1}{2}}V_0$, $K_1$, and $C_1$ are compact. If $\ker B\neq\{0\}$ then $\ac>-d^2/4$ follows from \eqref{3Actable}. Hence, $K=0$, Lemma \ref{2schurlem}, and Lemma \ref{3invlemc} yields that $I_{\Hs_1}+K_1$ is invertible. The claim then follows from Theorem \ref{main2}.

Lemma \ref{ratinf} and Lemma \ref{mobrooti} implies that there is a branch of accumulating eigenvalues to $\infty$ of $T$ if and only if there is a branch of accumulating eigenvalues to $0$ of \eqref{pentinf}. Since \eqref{pentinf} has the structure of \eqref{Adef2} with $k=2$ and $\ker H'=\{0\}$ the result follows by the same arguments as for the accumulation to $-id/2$.
\end{proof}

\begin{center}
\begin{figure}
\includegraphics[width=12cm]{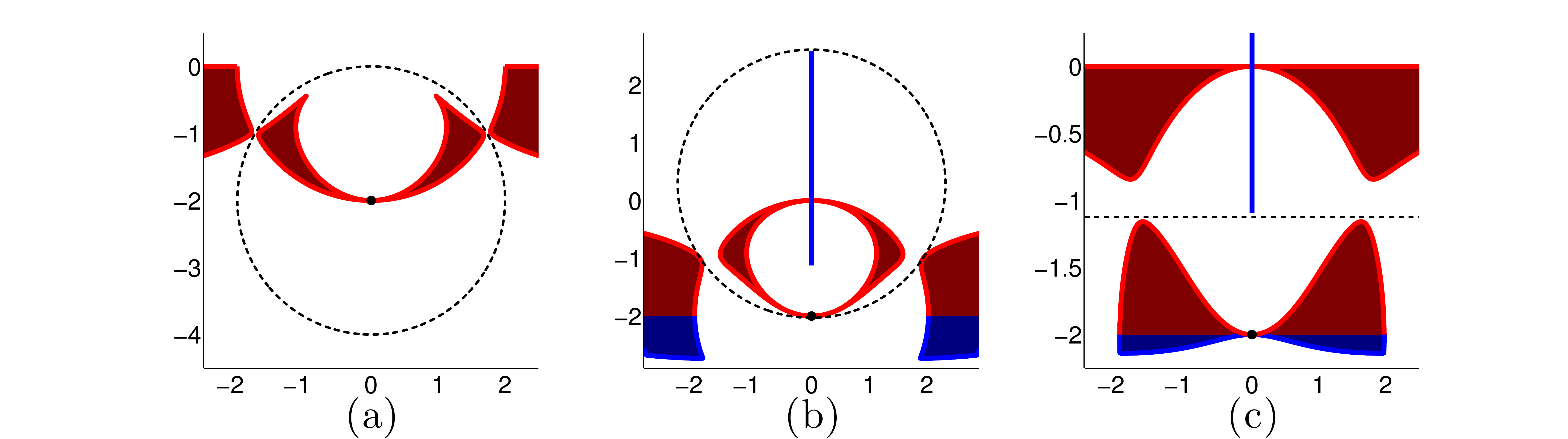}
\caption{Examples of $W_\Omega(T)$ in red and blue, with $c=d^2/4=4$, where the dashed circle separates a component from the rest of $W_\Omega(T)$. In (a) $d^2/4\in \overline{W(B)}$.  In (b) $\beta_0>d^2/4$.  In (c) $\beta_1<d^2/4$.  }\label{3fig:disjc}
\end{figure}
\end{center}

In the case $d^2/4\notin \overline{W(B)}$ the lower bound given in \eqref{3Actable} is the smallest value such that a circle $\Gh$ that separate the component exists. For a smaller $\ac$ the components might still be disjoint but there will be no circle separating them.

Figure \ref{3fig:disjc} illustrate cases where we can show the existence of a spectral divisor of order two. The dashed line in Figure \ref{3fig:disjc}.(c) is the boundary of the disc $\Gh$ containing the component with the pole. 

\begin{rem}
Since we know that the  sequence of accumulating eigenvalues is located in $W_\Omega(T)$, properties on the angle of the accumulation and similar results can be seen from the behavior of the set $W_\Omega(T)$ close to the pole. These properties of $W_\Omega(T)$ are illustrated in Figure \ref{3fig:disjp}, Figure \ref{3fig:disjm}, and in Figure \ref{3fig:disjc}.
\end{rem}

\begin{rem}

Sufficient conditions on $\alc$ for accumulation of eigenvalues to complex $\infty$ can also be found in the case $d\neq2\sqrt{c}$. However, in the transformed problem $\infty$ is mapped to a double root while the poles are simple. Hence, the sufficient conditions on $\alc$ in Theorem  \ref{3lemallB} and in Theorem \ref{3lemallBi}, will in general not be related to accumulation of eigenvalues to $\infty$.
\end{rem}

\section{Application to absorptive photonic crystals}
\label{Application}

In this section, we consider a rational operator function with applications in modelling propagation of electromagnetic waves in periodic structures, such as absorptive photonic crystals and metamaterials \cite{0305-4470-33-35-311,MR2164712,MR2718134,MR3023431}. The time-evolution is governed by Maxwell's equations with time-dependent coefficients. Using the Fourier-Laplace transform, we then obtain the stationary Maxwell equations, where the non-magnetic material properties are characterized by a space $x$ and frequency $\omega$ dependent permittivity function $\epsilon$ \cite{MR1409140}.

The problem studied in this section is a generalization of \cite{MR3543766}, where a selfadjoint operator function was studied. Here, we apply the theory developed in the previous sections to a non-selfadjoint rational operator function with periodic permittivity $\epsilon$. This enables us to consider multi-pole Lorentz models of $\epsilon$ in full generality  \cite{MR1409140} and we study an unbounded operator function that is used to determine Bloch solutions \cite[Chapter 3.1]{MR1232660}. 

Let $\Gamma$ denote the lattice $\mathbb{Z}^{n}$ and $\Omega:=(0,1]^{n}$ the unit cell of the lattice $\Gamma$.
In most applications the function $\epsilon$ is piecewise constant in $x\in\Omega$ and we let $\Omega  =\Omega_1\cup \cdots\cup \Omega_M$, $M\in\N$, denote a partitioning of $\Omega= (0,1]^n$. Let $\chi_{\Omega_m}$ denote the characteristic function of the subset $\Omega_m \subset (0,1]^n$ and define for given $a_m>0$ the operator
\[
W:\Ltwo\rightarrow\Ltwo, \quad W:=\sum_{m=1}^{M}a_m \chi_{\Omega_m}. 
\]
The material properties are then characterized by the multi-pole Lorentz model $\omega^2\epsilon(\cdot, \omega):\Ltwo\rightarrow\Ltwo$,
\begin{equation}\label{eq:epsSum}
	\omega^2\epsilon(\cdot, \omega) := -W\omega^2-\!\sum_{m=1}^{\wh{M}}\sum_{\ell=1}^{L_m} \frac{\omega^2\bb_{m,\ell}}{\cc_{m,\ell}-id_{m,\ell}\omega-\omega^2}\chi_{\Omega_m} (\cdot), \quad\omega\in\domain,
\end{equation}
with $\wh M \in \{1,2,\dots,M\}$, $\bb_{m,\ell}>0$, $\cc_{m,\ell}\geq 0$, $d_{m,\ell}\geq0$, and $\domain$ denotes the set of all $\omega\in\C$ that are not poles of \eqref{eq:epsSum}.  The case when $d_{m,\ell}=0$ for all $m$ and $n$ was studied in \cite{MR3543766}.

The dual lattice to $\Gamma$ is $\Gamma^{*}:=2\pi\Z^2$ and we define the Brillouin zone of the dual lattice $\Gamma^{*}$ as the set $\CK:=(-\pi,\pi]^{n}$. For fixed $k\in\CK=(-\pi,\pi]^n$ the shifted Laplace operator $\Delta_k:\Ltwo\rightarrow\Ltwo$ is defined as 
\begin{equation}
	-\Delta_k:=\sum_{j=1}^n \left (i\frac{\partial}{\partial x_j}-k_j \right )^2,
	\quad \dom \Delta_k=H^{2}(\T^n)
\end{equation}
and we consider spectral properties of  $\widehat T_k(\omega):\Ltwo\rightarrow\Ltwo$,
\begin{equation}
	\widehat T_k(\omega):=-\Delta_k-\omega^2\epsilon(\cdot,\omega),\quad k\in\CK=(-\pi,\pi]^n,\quad \omega\in\domain,
	\label{eq:bas}
\end{equation}
where $\dom \widehat T_k(\omega)= H^{2}(\T^2)$ for all $\omega\in\domain$. This function has after scaling with $W$ the form \eqref{eq:Tdef} and we define therefore the operator $T_k(\omega):\Ltwo\rightarrow\Ltwo$ by
\begin{equation}\label{eq:Sk}
\begin{aligned}
	T_k(\omega) 	& :=W^{-\frac 12}\widehat T_k(\omega)W^{-\frac 12}
					= A_k-\omega^2-\!\sum_{m=1}^{\wh{M}}\sum_{\ell=1}^{L_m} \frac{\omega^2}{\cc_{m,\ell}-id_{m,\ell}\omega-\omega^2}B_{m,\ell}, 
\end{aligned} 
\end{equation}
where 
\begin{equation}\label{eq:A0+A1}
  A_k\!:=-W^{-\frac 12}\Delta_k W^{-\frac 12}\!,\quad B_{m,\ell}\!:=\frac{\bb_{m,\ell}}{a_m}\chi_{\Omega_m}
\end{equation}
and the domain of $T_k$ is 
\[
  \dom T_k= \dom A_k = W^{\frac 12}H^{2}(\T^2). 
\]
The operator $\Delta_k$ has a compact resolvent and $\sigma (-\Delta_k)=\{|2\pi j +k|^n\,:\,j\in \Z^n\}$, \cite[p.\ 161-164]{MR1232660}. Hence $A_k$ is selfadjoint with discrete spectrum and $B_{m,\ell}$ is a bounded selfadjoint operator. Moreover, we have the estimates 
\[
	A_k\geq \frac{|k|^n}{\max\{a_m\}},\quad 0\leq B_{m,\ell} \leq \frac{\bb_{m,\ell}}{a_m}
\]
and it is clear that it exists a real $\alpha_0$ such that $(A_k-\alpha_0)^{-1}\in S^p(\Ltwo)$ for some $p>1$.

The operator function $T_k$ is of the form \eqref{eq:Tdefex} and Theorem \ref{prop:inf} implies that for $A_k$ sufficiently large there exists a branch of eigenvalues accumulating at $\infty$. Moreover, Theorem \ref{prop:mult} implies that there are branches of complex eigenvalues that accumulate to each of the poles, provided that $A_k$ is sufficiently large. In particular if $\widehat{M}=1$ and $L_1=1$ then Theorem \ref{3lemallB}, Theorem \ref{3lemallBi} or Theorem \ref{3lemallBt} (depending on the case) can be used to find a sufficient lower bound on $A_k$. These propositions also state sufficient conditions for completeness and minimality of the set of eigenvectors and associated vectors corresponding to a branch of eigenvalues accumulating at one of the poles. 

In Figure \ref{fig:clustpic} we present an example where $\widehat{M}=1$, $L_1=1$, $k=(\sqrt{3}\pi/2,\pi/2)$, $W=I_\Hs$, $b_{1,1}=50$, $c_{1,1}=9$,  $d_{1,1}=2$, and $n=2$. Consequently, $\overline{W(A_k)}=[\pi^2,\infty)$ and  $\overline{W(B_{1,1})}=[0,50]$. Since the parameters fulfill the conditions $d_{1,1}^2/4\in\overline{W(B_{1,1})}$ and  $c_{1,1}<\overline{W(A_k)}$,  Theorem \ref{3lemallB} implies that branches of complex eigenvalues exists that accumulate to the poles of $T_k$.  Figure \ref{fig:clustpic} also depicts numerically computed eigenvalues, where the method presented in \cite{MR2876569} was used.

\begin{centering}
\begin{figure}
\includegraphics[width=12cm]{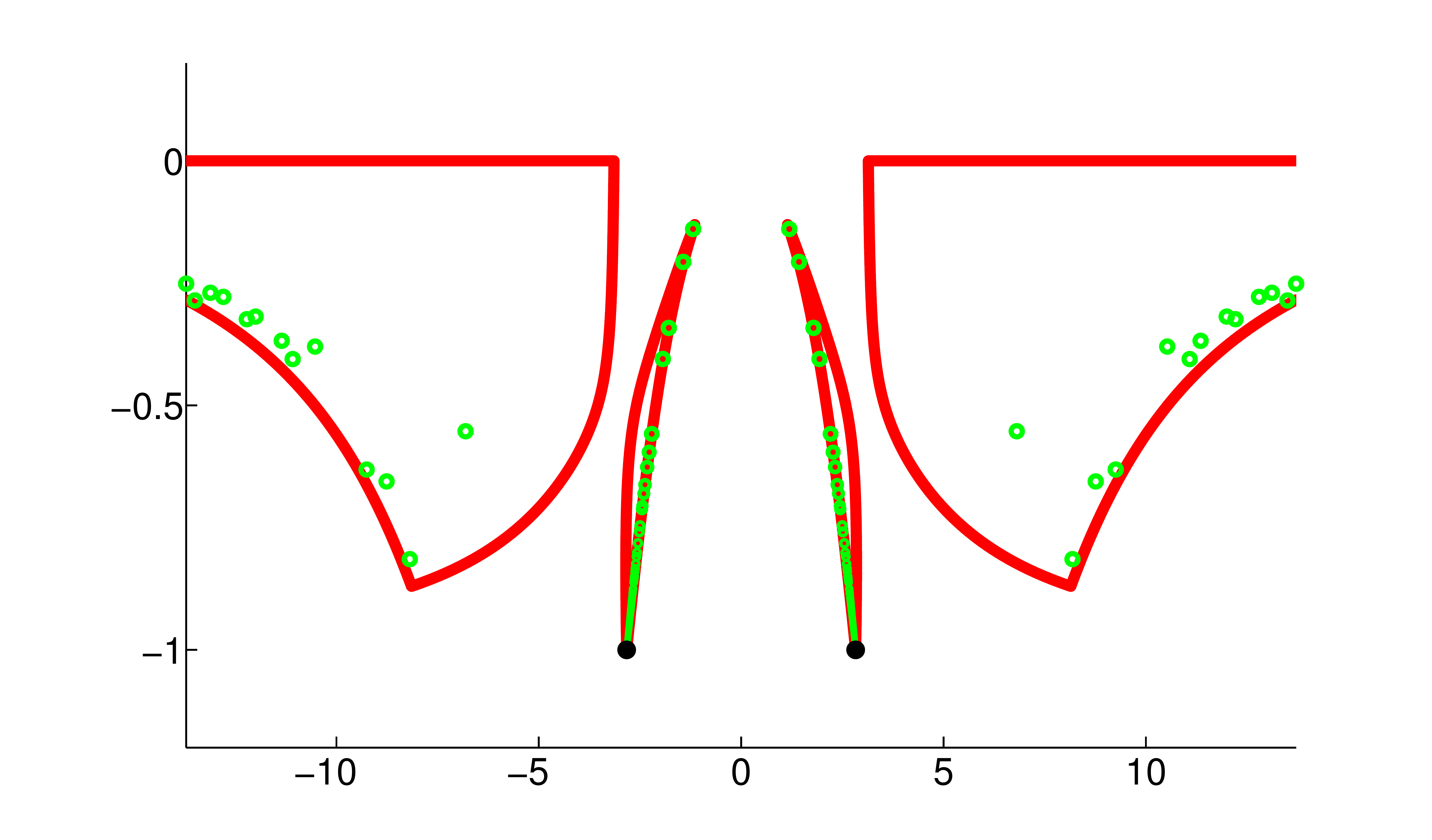}
\caption{Visualization of accumulation of eigenvalues for $T_{k}$, when $\widehat{M}=1$, $L_1=1$, $k=(\sqrt{3}\pi/2,\pi/2)$, $W=I_\Hs$, $b_{1,1}=50$, $c_{1,1}=9$,  $d_{1,1}=2$, and $n=2$. The boundary of $W_\Omega(T_k)$ is given by the solid lines, the circles are numerically computed eigenvalues, and the dots are the poles $\pm\sqrt{8}-i$. }\label{fig:clustpic}
\end{figure}
\end{centering}

\vspace{3.5mm}

{\small
{\bf Acknowledgements.} \ 
The authors gratefully acknowledge the support of the Swedish Research Council under Grant No.\ $621$-$2012$-$3863$. 
}

\bibliographystyle{alpha}
\bibliography{bibliography}

\newcommand{\etalchar}[1]{$^{#1}$}
\begin{thebibliography}{SEK{\etalchar{+}}05}

\bibitem[AAD01]{MR1819914}
A.~A. Abramov, A.~Aslanyan, and E.~B. Davies.
\newblock Bounds on complex eigenvalues and resonances.
\newblock {\em J. Phys. A}, 34(1):57--72, 2001.

\bibitem[APT02]{MR1911850}
V.~Adamjan, V.~Pivovarchik, and C.~Tretter.
\newblock On a class of non-self-adjoint quadratic matrix operator pencils
  arising in elasticity theory.
\newblock {\em J. Operator Theory}, 47(2):325--341, 2002.

\bibitem[B{\"o}g17]{MR3627408}
S.~B{\"o}gli.
\newblock Schr\"odinger operator with non-zero accumulation points of complex
  eigenvalues.
\newblock {\em Comm. Math. Phys.}, 352(2):629--639, 2017.

\bibitem[Ces96]{MR1409140}
M.~Cessenat.
\newblock {\em Mathematical methods in electromagnetism}, volume~41 of {\em
  Series on Advances in Mathematics for Applied Sciences}.
\newblock World Scientific Publishing Co. Inc., River Edge, NJ, 1996.

\bibitem[CMM12]{MR3023431}
P-H. Cocquet, P-A. Mazet, and V.~Mouysset.
\newblock On the existence and uniqueness of a solution for some
  frequency-dependent partial differential equations coming from the modeling
  of metamaterials.
\newblock {\em SIAM J. Math. Anal.}, 44(6):3806--3833, 2012.

\bibitem[{\c{C}}ol08]{MR2397852}
N.~{\c{C}}olako{\u g}lu.
\newblock The numerical range of a class of self-adjoint operator functions.
\newblock In {\em Recent advances in matrix and operator theory}, volume 179 of
  {\em Oper. Theory Adv. Appl.}, pages 145--155. Birkh\"auser, Basel, 2008.

\bibitem[DS88]{DUSC}
N.~Dunford and J.~T. Schwartz.
\newblock {\em Linear operators. {P}art {I}{}}.
\newblock Wiley Classics Library. John Wiley \& Sons, Inc., New York, 1988.

\bibitem[EKE12]{MR2876569}
C.~Effenberger, D.~Kressner, and C.~Engstr{\"o}m.
\newblock Linearization techniques for band structure calculations in absorbing
  photonic crystals.
\newblock {\em Internat. J. Numer. Methods Engrg.}, 89(2):180--191, 2012.

\bibitem[ELT17]{MR3543766}
C.~Engstr\"om, H.~Langer, and C.~Tretter.
\newblock Rational eigenvalue problems and applications to photonic crystals.
\newblock {\em J. Math. Anal. Appl.}, 445(1):240--279, 2017.

\bibitem[Eng10]{MR2718134}
C.~Engstr{\"o}m.
\newblock On the spectrum of a holomorphic operator-valued function with
  applications to absorptive photonic crystals.
\newblock {\em Math. Models Methods Appl. Sci.}, 20(8):1319--1341, 2010.

\bibitem[ET17a]{ATEN}
C.~Engstr\"om and A.~Torshage.
\newblock Enclosure of the numerical range of a class of non-selfadjoint
  rational operator.
\newblock {\em Integral Equations and Operator Theory}, 88(2):151--184, 2017.

\bibitem[ET17b]{ATEQ}
C.~Engstr\"om and A.~Torshage.
\newblock On equivalence and linearization of operator matrix functions with
  unbounded entries.
\newblock {\em Integral Equations and Operator Theory}, 89(4):465--492, 2017.

\bibitem[FLS16]{MR3556444}
R.~L. Frank, A.~Laptev, and O.~Safronov.
\newblock On the number of eigenvalues of {S}chr\"odinger operators with
  complex potentials.
\newblock {\em J. Lond. Math. Soc. (2)}, 94(2):377--390, 2016.

\bibitem[Fra11]{MR2820160}
R.~L. Frank.
\newblock Eigenvalue bounds for {S}chr\"odinger operators with complex
  potentials.
\newblock {\em Bull. Lond. Math. Soc.}, 43(4):745--750, 2011.

\bibitem[GKL78]{MR0482317}
I.~C. Gohberg, M.~A. Kaashoek, and D.~C. Lay.
\newblock Equivalence, linearization, and decomposition of holomorphic operator
  functions.
\newblock {\em J. Funct. Anal.}, 28(1):102--144, 1978.

\bibitem[Han13]{MR3054310}
M.~Hansmann.
\newblock Variation of discrete spectra for non-selfadjoint perturbations of
  selfadjoint operators.
\newblock {\em Integral Equations Operator Theory}, 76(2):163--178, 2013.

\bibitem[KL78]{KL78}
M.~G. Kre{\u\i }n and H.~Langer.
\newblock On some mathematical principles in the linear theory of damped
  oscillations of continua. {I}, {II}.
\newblock {\em Integral Equations Operator Theory}, 1:364--399, 539--566, 1978.

\bibitem[Kuc93]{MR1232660}
P.~Kuchment.
\newblock {\em Floquet theory for partial differential equations}, volume~60 of
  {\em Operator Theory: Advances and Applications}.
\newblock Birkh\"auser Verlag, Basel, 1993.

\bibitem[KVL92]{MR1155350}
M.~A. Kaashoek and S.~M. Verduyn~Lunel.
\newblock Characteristic matrices and spectral properties of evolutionary
  systems.
\newblock {\em Trans. Amer. Math. Soc.}, 334(2):479--517, 1992.

\bibitem[LMM06]{MR2216445}
H.~Langer, A.~Markus, and V.~Matsaev.
\newblock Self-adjoint analytic operator functions and their local spectral
  function.
\newblock {\em J. Funct. Anal.}, 235(1):193--225, 2006.

\bibitem[LMM12]{MR2931941}
H.~Langer, A.~Markus, and V.~Matsaev.
\newblock Linearization, factorization, and the spectral compression of a
  self-adjoint analytic operator function under the condition ({VM}).
\newblock In {\em A panorama of modern operator theory and related topics},
  volume 218 of {\em Oper. Theory Adv. Appl.}, pages 445--463.
  Birkh\"auser/Springer Basel AG, Basel, 2012.

\bibitem[LS09]{MR2540070}
A.~Laptev and O.~Safronov.
\newblock Eigenvalue estimates for {S}chr\"odinger operators with complex
  potentials.
\newblock {\em Comm. Math. Phys.}, 292(1):29--54, 2009.

\bibitem[LS16]{MR3460233}
M.~Langer and M.~Strauss.
\newblock Triple variational principles for self-adjoint operator functions.
\newblock {\em J. Funct. Anal.}, 270(6):2019--2047, 2016.

\bibitem[Mar88]{ASM}
A.~S. Markus.
\newblock {\em Introduction to the spectral theory of polynomial operator
  pencils}, volume~71 of {\em Translations of Mathematical Monographs}.
\newblock American Mathematical Society, Providence, RI, 1988.

\bibitem[Pav66]{MR0203530}
B.~S. Pavlov.
\newblock On a non-selfadjoint {S}chr\"odinger operator.
\newblock In {\em Problems of {M}athematical {P}hysics, {N}o. {1}, {S}pectral
  {T}heory and {W}ave {P}rocesses ({R}ussian)}, pages 102--132. Izdat.
  Leningrad. Univ., Leningrad, 1966.

\bibitem[Pav67]{MR0234319}
B.~S. Pavlov.
\newblock On a non-selfadjoint {S}chr\"odinger operator. {II}.
\newblock In {\em Problems of {M}athematical {P}hysics, {N}o. 2, {S}pectral
  {T}heory, {D}iffraction {P}roblems ({R}ussian)}, pages 133--157. Izdat.
  Leningrad. Univ., Leningrad, 1967.

\bibitem[Pav68]{MR0348554}
B.~S. Pavlov.
\newblock On a nonselfadjoint {S}chr\"odinger operator. {III}.
\newblock In {\em Problems of {M}athematical {P}hysics, {N}o. 3: {S}pectral
  theory ({R}ussian)}, pages 59--80. Izdat. Leningrad. Univ., Leningrad, 1968.

\bibitem[RS04]{MR2038194}
I.~Rodnianski and W.~Schlag.
\newblock Time decay for solutions of {S}chr\"odinger equations with rough and
  time-dependent potentials.
\newblock {\em Invent. Math.}, 155(3):451--513, 2004.

\bibitem[Sam17]{MR3682742}
D.~Sambou.
\newblock On eigenvalue accumulation for non-self-adjoint magnetic operators.
\newblock {\em J. Math. Pures Appl. (9)}, 108(3):306--332, 2017.

\bibitem[SEK{\etalchar{+}}05]{MR2164712}
D.~Sj\"oberg, C.~Engstr\"om, G.~Kristensson, D.~J.~N. Wall, and N.~Wellander.
\newblock A {F}loquet-{B}loch decomposition of {M}axwell's equations applied to
  homogenization.
\newblock {\em Multiscale Model. Simul.}, 4(1):149--171, 2005.

\bibitem[SP80]{Sanchez-Palencia}
E.~Sanchez-Palencia.
\newblock {\em {Non-Homogeneous Media and Vibration Theory}}.
\newblock Lecture Notes in Physics. Springer, Berlin, 1980.

\bibitem[TMC00]{0305-4470-33-35-311}
A.~Tip, A.~Moroz, and J.~M. Combes.
\newblock Band structure of absorptive photonic crystals.
\newblock {\em Journal of Physics A: Mathematical and General}, 33(35):6223,
  2000.

\end{thebibliography}

\vspace{-2mm}

\end{document}